\documentclass[11pt,a4paper]{article}
\usepackage{amsfonts}
\usepackage{latexsym}
\usepackage{cite}
\usepackage{amsmath,amsfonts,latexsym,amssymb}
\usepackage[mathscr]{eucal}
\usepackage{cases,color}
\usepackage{amsthm}
\usepackage[bf,small]{caption2}
\usepackage{float}
\usepackage{graphicx}
\usepackage{amsmath}
\usepackage{amssymb}
\usepackage[all]{xy}

\newtheorem{theorem}{theorem}[section]

\newtheorem{cor}[theorem]{Corollary}
\newtheorem{conv}[theorem]{Convention}
\newtheorem{defn}[theorem]{Definition}
\newtheorem{exmp}[theorem]{Example}
\newtheorem{lem}[theorem]{Lemma}
\newtheorem{nota}[theorem]{Notation}

\newtheorem{rmk}[theorem]{Remark}
\newtheorem{thm}[theorem]{Theorem}

\begin{document}

\title{\vspace{-2cm}\textbf{Computing Alexander polynomials for arborescent links}}
\author{\Large Haimiao Chen}

\date{}
\maketitle

\begin{abstract}
  Alexander polynomial was born one century ago, but explicit formulas have been found for only a few families of links.
  In this paper, we present an efficient method of computing Alexander polynomial for arborescent links.
  Applying this method, we express the Alexander polynomials of Montesinos links in terms of certain polynomials associated to rational tangles which can be computed recursively. Specifically, we deduce explicit closed formulas for all pretzel links.

\end{abstract}

\section{Introduction}

Alexander polynomial is a traditional knot invariant, dating back to 1928 \cite{Al28}, and has been found to be widely useful in knot theory.
However, until now it has been explicitly computed for only a few families of links.
The formula for torus knots are well-known.
An elegant formula for 2-bridge knots was found by Hartley \cite{Ha79}, and extended to links by Hoste \cite{Ho20}. Alexander polynomials of 2-bridge links were also studied by Kanenobu \cite{Ka84}. Formulas for some pretzel links were obtained in \cite{BL20,Hi01,KL07,Na86}; in particular, a finite procedure of computation was developed in \cite{Na86}. Very recently, Belousov \cite{Be25} deduced explicit formulas for all pretzel knots.

In this paper, we propose an efficient and easy-to-grasp method for computing the multi-variable Alexander polynomials of arborescent links, which form an interesting class including 2-bridge links and pretzel links. For arborescent links, we prefer combinatorial notations to the traditional one using trees \cite{BS16}. A step-by-step method was proposed by Hirasawa and Murasugiin \cite{HM19}, based on manipulating Seifert matrices, but it has not been applied to deduce concrete results. Differing from the usual approaches of Seifert matrices or skein relations, we fully utilize the combinatorial properties of link diagrams.

Referring to \cite[Page 116--119]{Li97}, we recall one of the definitions of the (multi-variable) Alexander polynomial $\Delta_L$ for an oriented link $L$. Suppose $L$ has components $K_1,\ldots,K_m$.
Let $x_1,\ldots,x_n$ denote the directed arcs of $L$, and $\mathfrak{c}_1,\ldots,\mathfrak{c}_n$ the crossings;
suppose $x_j$ belongs to $K_{\nu(j)}$.
In the Wirtinger presentation for $\pi(L):=\pi_1(S^3\setminus L)$,
each $x_j$ provides a generator, denoted also by $x_j$, and each $\mathfrak{c}_i$ provides a relator $r_i$.
Let $F_n$ denote the free group generated by $x_1,\ldots,x_n$.
Let
$$\Phi:\mathbb{Z}[F_n]\to\mathbb{Z}[t_1^{\pm1},\ldots,t_m^{\pm1}]$$
denote the ring homomorphism determined by $x_j\mapsto t_{\nu(j)}$.
Let $M$ denote the $n\times n$ matrix with the $(i,j)$-entry $\Phi(\partial r_i/\partial x_j)$, where $\partial r_i/\partial x_j$ is the Fox derivative.
Arbitrarily choose $i_0,j_0\in\{1,\ldots,n\}$, and let $M'$ be the matrix obtained by deleting the $i_0$-th row and $j_0$-th column of $M$. Then
$$\Delta_{L}\doteq\begin{cases} \det(M'),&m=1 \\  \det(M')/(1-t_{\nu(j_0)}),&m\ge2 \end{cases},$$
which turns out to be independent of the choices of $i_0$ and $j_0$; here $\doteq$ means equality up to a factor of the form $\pm t_1^{k_1}\cdots t_m^{k_m}$. Denote $t_1$ as $t$ if $m=1$.

The content is organized as follows.
In Section 2, we present a method in Theorem \ref{thm:main}, for computing $\Delta_L$ when $L$ is an arborescent link. As building blocks of formulas, to each rational tangle $[p/q]$ we associate a pair of polynomials $z_v(p/q)$, $z_h(p/q)$ which can be recursively computed.
In Section 3, we apply the method to some families of links. When $L$ is the $(p_1/q_1,\ldots,p_r/q_r)$ Montesinos link, we manage to express $\Delta_L$ in terms of the polynomials $z_v(p_i/q_i)$, $z_h(p_i/q_i)$. Specifically, we deduce explicit formulas for all pretzel links;
in particular, we reprove the formulas for pretzel knots given in \cite{Be25}, justifying the correctness of the method.
In Section 4, we prove Theorem \ref{thm:main}.

\section{The method}

Let $\mathcal{T}_2^2$ denote the set of four-end tangles of the form shown at leftmost in Figure \ref{fig:arborescent}.
To each $T\in\mathcal{T}_2^2$ are associated two links, called the {\it numerator} $N(T)$ and the {\it denominator} $D(T)$.
Defined on $\mathcal{T}_2^2$ are the vertical composition $\ast$ and the horizontal composition $+$.

\begin{figure}[H]
  \centering
  \includegraphics[width=12.5cm]{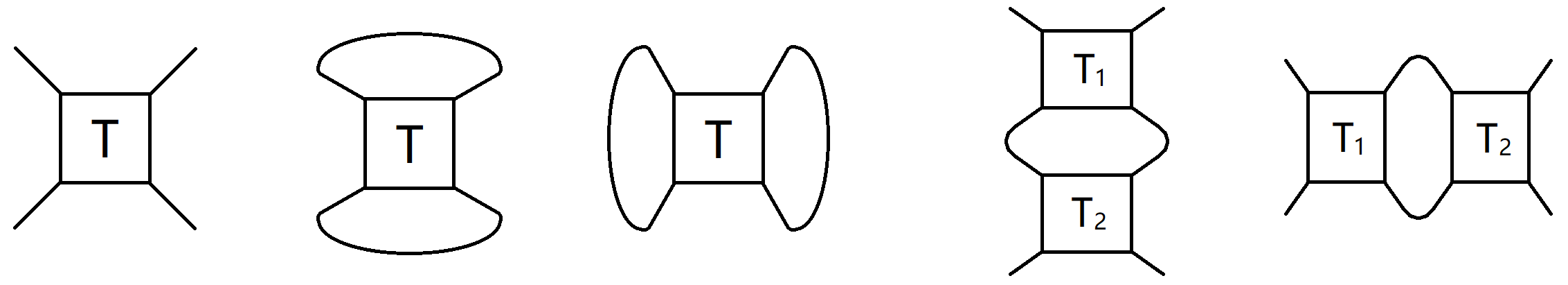}\\
  \caption{From left to right: a tangle $T\in\mathcal{T}_2^2$; $N(T)$; $D(T)$; $T_1\ast T_2$; $T_1+T_2$.}\label{fig:arborescent}
\end{figure}

A tangle $T\in\mathcal{T}_2^2$ is called {\it arborescent} if it can be constructed from copies of $[\pm1]$ (see Figure \ref{fig:basic}) by repeatedly applying $\ast$ and $+$. Let $\mathcal{T}_{\rm ar}$ denote the set of arborescent tangles.
An {\it arborescent link} is one of the form $N(T)$ or $D(T)$ for some $T\in\mathcal{T}_{\rm ar}$.

Since $N(T)=D(T^{\rm rot})$, where $T^{\rm rot}$ results from rotating $T$ by $\pi/2$, we may focus on links of the form $D(T)$.

\begin{figure}[H]
  \centering
  \includegraphics[width=4cm]{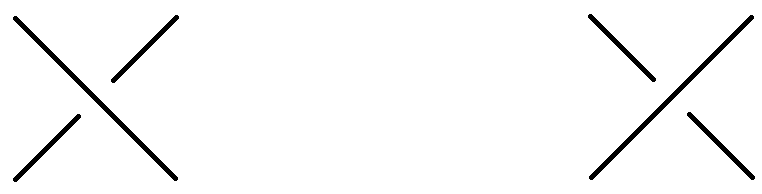}\\
  \caption{Left: $[1]$. Right: $[-1]$.}\label{fig:basic}
\end{figure}

For $k\ne 0$, the horizontal composite of $|k|$ copies of $[1]$ (resp. $[-1]$) is denoted by $[k]$ if $k>0$ (resp. $k<0$), and the
vertical composite of $|k|$ copies of $[1]$ (resp. $[-1]$) is denoted by $[1/k]$ if $k>0$ (resp. $k<0$).
Given $p/q\in\mathbb{Q}$, if a continued fraction of $p/q$ is
$$[[k_1,\ldots,k_s]]:=k_s+1/(k_{s-1}+\cdots+1/k_1)\cdots),$$
then the associated {\it rational tangle} is defined as
$$[p/q]=[[k_1],\ldots,[k_s]]=\begin{cases}
(\cdots([k_{1}]\ast[1/k_{2}])+\cdots)+[k_{s}], &2\nmid s \\
(\cdots([1/k_{1}]+[k_{2}])\ast\cdots)+[k_{s}], &2\mid s
\end{cases}.$$
The notation $[p/q]$ is justified by a result of Conway (see \cite[Proposition 2.1]{FP16}), which asserts that up to isotopy the tangle depends only on $p/q$. We call $[k_s]$ the {\it integer part} of $[p/q]$.

When speaking of $T\in\mathcal{T}_{\rm ar}$, we usually assume that an order of repeated horizontal/vertical compositions of rational tangles has been chosen. Call the arborescent subtangles appearing in the compositions {\it subsequent to} $T$. 
For example, if $R_1,\ldots,R_6$ are rational tangles, then
$$T=((R_1+R_2)+(R_3\ast R_4))\ast (R_5+R_6)\in\mathcal{T}_{\rm ar},$$
and $R_3$, $R_3\ast R_4$, $(R_1+R_2)+(R_3\ast R_4)$, $R_5$, etc. are subsequent to $T$.

Given a tangle $T$, let $\mathcal{D}(T)$ denote the set of directed arcs; each arc of $T$ gives rise to two directed arcs.
For $\mathsf{a}\in\mathcal{D}(T)$, let $\mathsf{a}^{-1}$ denote the directed arc obtained by reversing the direction of $\mathsf{a}$.
When $T\in\mathcal{T}_{\rm ar}$, let $T^{{\rm nw}}$, $T^{{\rm ne}}$, $T^{{\rm sw}}$, $T^{{\rm se}}$ respectively denote the arc at the northwest, northeast, southwest, southeast end of $T$, all directed outward.

\begin{defn}
\rm Let $J=\{t_1^{\pm1},t_2^{\pm1},\ldots,\}$. A {\it coloring} of a tangle $T$ is a map $\alpha:\mathcal{D}(T)\to J$ such that $\alpha(\mathsf{a}^{-1})=\alpha(\mathsf{a})^{-1}$ for each $\mathsf{a}$, and $\alpha(\mathsf{a})=\alpha(\mathsf{b})$ if $\mathsf{a},\mathsf{b}$ belong to the same component of $T$.
Call $\alpha(\mathsf{a})$ the {\it color} of $\mathsf{a}$. Call $(T,\alpha)$ a {\it colored tangle}. We will omit $\alpha$ if it is clear from the context.
\end{defn}

If $D(T)$ is an oriented link, then $\Phi$ equips $T$ with a coloring.

Let $\mathcal{T}_{\rm ar}^c$ denote the set of colored arborescent tangles.
For $T\in\mathcal{T}_{\rm ar}^c$, we always use $\alpha$ to denote its coloring if it is implicit.

Define an equivalence relation $\sim$ on $\mathbb{Z}[J]\times\mathbb{Z}[J]$ by declaring $(f_1,g_1)\sim(f_2,g_2)$ if there exists
$\kappa=\epsilon t_{i_1}^{n_1}\cdots t_{i_r}^{n_r}$ with $n_1,\ldots,n_r\in\mathbb{Z}$, $\epsilon\in\{\pm1\}$ such that $f_2=\kappa f_1$ and $g_2=\kappa g_1$. Let $\Omega=(\mathbb{Z}[J]\times\mathbb{Z}[J])/\sim$.
Denote the equivalence class of $(f,g)$ by $[f:g]$. Clearly, if $g\ne 0$, then $f/g$ depends only on $[f:g]$.

\begin{nota}
\rm For $k\in\mathbb{Z}$ and a unit $a$ in some commutative ring, put
$$[k]_a=\begin{cases} 1+a+\cdots+a^{k-1}, &k>0 \\ 0, &k=0 \\ -a^k(1+a+\cdots+a^{|k|-1}), &k<0\end{cases}.$$
\end{nota}

We associate to each $T\in\mathcal{T}_{\rm ar}^c$ an element $z(T)=[z_v(T):z_h(T)]\in\Omega$, recursively as follows:
\begin{enumerate}
  \item[\rm(R1)] For $k\in\mathbb{Z}$, let $t_{\rm ne}=\alpha([k]^{\rm ne})$, etc., and put $z([k])=[z_v([k]):1]$, with
                \begin{align*}
                z_v([k])=\begin{cases} t_{\rm ne}(t_{\rm se}-1)[h]_{t_{\rm ne}t_{\rm se}},&k=2h  \\
                (1-t_{\rm se})[h]_{t_{\rm ne}t_{\rm se}}-1,&k=2h-1 \end{cases}.
                \end{align*}
                Let $r_{\rm ne}=\alpha([1/k]^{\rm ne})$, etc., and put $z([1/k])=[1:z_h([1/k])]$, with
                \begin{align*}
                z_h([1/k])=\begin{cases} r_{\rm sw}^{-1}(r_{\rm se}^{-1}-1)[h]_{r_{\rm sw}^{-1}r_{\rm se}^{-1}},&k=2h  \\
                (1-r_{\rm se}^{-1})[h]_{r_{\rm sw}^{-1}r_{\rm se}^{-1}}-1,&k=2h-1 \end{cases}.
                \end{align*}
                In particular, $z([-1])=[-1:1]$, $z([1])=[-t_{\rm se}:1]$, with $t_{\rm se}=\alpha([1]^{\rm se})$.
  \item[\rm(R2)] Let $t_{\rm ne}$, $t_{\rm se}$, $t_{\rm sw}$ respectively denote the colors of $T_1^{\rm ne}$, $T_1^{\rm se}$, $T_1^{\rm sw}$.

                For $T=T_1\ast T_2$, put $z_v(T)\doteq z_v(T_1)z_v(T_2)$, and
                $$z_h(T)=z_h(T_1)z_v(T_2)+\frac{1-t_{\rm ne}^{-1}}{1-t_{\rm se}}z_h(T_2)z_v(T_1)
                +\frac{t_{\rm sw}-t_{\rm se}^{-1}}{1-t_{\rm se}^{-1}}z_h(T_1)z_h(T_2);$$
                for $T=T_1+T_2$, put $z_h(T)\doteq z_h(T_1)z_h(T_2)$, and
                $$z_v(T)=z_v(T_1)z_h(T_2)+\frac{1-t_{\rm sw}}{1-t_{\rm se}^{-1}}z_v(T_2)z_h(T_1)
                +\frac{t_{\rm ne}^{-1}-t_{\rm se}}{1-t_{\rm se}}z_v(T_1)z_v(T_2).$$
\end{enumerate}

Given $\mathsf{t}=(t_{\rm ne},t_{\rm se},t_{\rm sw})$, define binary operations $\ast_{\mathsf{t}},\circ_{\mathsf{t}}$ by
\begin{align}
b_1\ast_{\mathsf{t}}b_2&=b_1+\frac{1-t_{\rm ne}^{-1}}{1-t_{\rm se}}b_2+\frac{t_{\rm sw}-t_{\rm se}^{-1}}{1-t_{\rm se}^{-1}}b_1b_2,   \label{eq:composition-b}  \\
c_1\circ_{\mathsf{t}}c_2&=c_1+\frac{1-t_{\rm sw}}{1-t_{\rm se}^{-1}}c_2+\frac{t_{\rm ne}^{-1}-t_{\rm se}}{1-t_{\rm se}}c_1c_2.    \label{eq:composition-c}
\end{align}
When $z_v(T_1)z_v(T_2)\ne 0$, the formula for $z_h(T)$ in (R2) can be rephrased as
\begin{align}
\frac{z_h(T_1\ast T_2)}{z_v(T_1\ast T_2)}=\frac{z_h(T_1)}{z_v(T_1)}\ast_{\mathsf{t}}\frac{z_h(T_2)}{z_v(T_2)};  \label{eq:fraction-v}
\end{align}
when $z_h(T_1)z_h(T_2)\ne 0$, the formula for $z_v(T)$ in (R2) can be rephrased as
\begin{align}
\frac{z_v(T_1+T_2)}{z_h(T_1+T_2)}=\frac{z_v(T_1)}{z_h(T_1)}\circ_{\mathsf{t}}\frac{z_v(T_2)}{z_h(T_2)}.   \label{eq:fraction-h}
\end{align}

The main result of the paper is
\begin{thm}\label{thm:main}
Suppose $L=D(T)$ is an oriented arborescent link, then $\Delta_{L}\doteq z_h(T)$ if $L$ is a knot, and
$\Delta_L\doteq z_h(T)/(1-\Phi(T^{\rm ne}))$ otherwise.
\end{thm}

\begin{nota}
\rm For $(T,\alpha)\in\mathcal{T}_{\rm ar}^c$, let $\sigma(T,\alpha)=(T^\sigma,\alpha^\sigma)$, where $T^\sigma$ denotes the tangle obtained by reflecting $T$ along its NW-SE diagonal, and $\alpha^\sigma(\mathsf{a}')=\alpha(\mathsf{a})^{-1}$ for each $\mathsf{a}'\in\mathcal{D}(T^\sigma)$ corresponding to $\mathsf{a}\in\mathcal{D}(T)$ under the reflection.
\end{nota}

\begin{lem}\label{lem:reflection}
$z(\sigma(T))=[z_h(T):z_v(T)]$ for each $T\in\mathcal{T}_{\rm ar}^c$.
\end{lem}

\begin{proof}
For $[k]$ and $[1/k]$ with $k\in\mathbb{Z}$, the assertion is true by the definition given in (R1).

Suppose the assertion holds for $T_1,T_2$.
Note that the colors of $\sigma(T_1)^{\rm ne}$, $\sigma(T_1)^{\rm se}$, $\sigma(T_1)^{\rm sw}$ are respectively the inverses of those of
$T_1^{\rm sw}$, $T_1^{\rm se}$, $T_1^{\rm ne}$. From the formulas in (R2) we can verify the assertion for $T_1\ast T_2$ and $T_1+T_2$.

By the recursive nature, $z(\sigma(T))=[z_h(T):z_v(T)]$ for each $T\in\mathcal{T}_{\rm ar}^c$.
\end{proof}

\begin{nota}
\rm Abbreviate $z_h([p/q])$ to $z_h(p/q)$, and $z_v([p/q])$ to $z_v(p/q)$.
\end{nota}

\begin{conv}\label{conv:zh}
\rm From now on, till the end of Section \ref{sec:computation}, we always set $z_h(k)=1$ for any $k\in\mathbb{Z}$, unless otherwise specified.
\end{conv}

\begin{figure}[H]
  \centering
  \includegraphics[width=11cm]{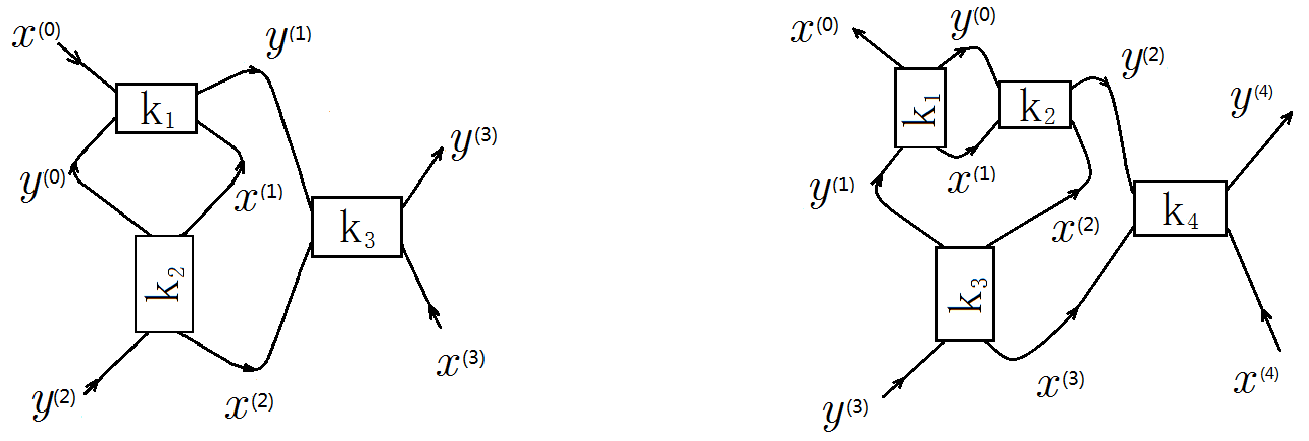}\\
  \caption{The rational tangle $[p/q]=[[k_1],\ldots,[k_s]]$, with $s=3$ on the left, and $s=4$ on the right. The arrows are used to choose a direction for each arc.}\label{fig:rational}
\end{figure}

\begin{lem}\label{lem:rational}
For a colored rational tangle $[p/q]=[[k_1],\ldots,[k_s]]$, as shown in Figure \ref{fig:rational}, let $u_i,v_i$ respectively denote the color of $x^{(i)}$, $y^{(i)}$. Let $t_1=u_0$, $t_2=v_0$, and for $1\le i\le s$, let
\begin{align*}
b_i=\begin{cases} v_{i}(u_i^{-1}-1)[h]_{v_{i}/u_i}, &k_{i}=2h \\
(1-v_{i})[h]_{v_{i}/u_i}-1,&k_{i}=2h-1 \end{cases}.
\end{align*}
Set $\eta_0=1$, $\eta_1=b_1$, and recursively compute $\eta_i$ via
\begin{align*}
\eta_{i}=\eta_{i-2}+\frac{1-v_{i-1}^{-1}}{1-u_{i-1}^{-1}}\eta_{i-1}b_i+\frac{v_{i-2}^{-1}-u_{i-1}}{1-u_{i-1}}\eta_{i-2}b_i.
\end{align*}
Then $z([p/q])=[z_v(p/q),z_h(p/q)]=[\eta_{s}:\eta_{s-1}]$.
\end{lem}

\begin{rmk}\label{rmk:rational}
\rm Although up to isotopy $[p/q]$ is determined by $p/q$, the functions $z_v(p/q)$, $z_h(p/q)$ do depend on the continued fraction $[k_1,\ldots,k_m]$.

Set $t_1=t_2=-1$, then $b_i=k_i$, so $\eta_i$ equals the numerator of $[[k_1,\ldots,k_i]]$. In particular, $\eta_s=p$, and $\eta_{s-1}=q$. Since $[[k_1],\ldots,[k_s]]$ has been defined, we have $\eta_i\ne 0$ for all $1\le i<s$.

Consequently, as elements of $\mathbb{Z}[J]$, $\eta_i\ne 0$ for all $1\le i<s$.
\end{rmk}

\begin{proof}
We use induction on $s$ to prove the assertion, which holds when $s=1$, since $z([k_1])=[b_1:1]$.

Suppose $s=n>1$ and that the assertion for $s=n-1$.

Let $H=[[k_1],\ldots,[k_{s-1}]]$.
Observe that $[p/q]=\sigma(H)+[k_s]$, so
\begin{align*}
z_h(p/q)\doteq z_h(\sigma(H))z_h(k_s)\doteq z_v(H)\doteq\eta_{s-1};
\end{align*}
we may write $z_h(p/q)=\kappa\eta_{s-1}$ for some unit $\kappa$.

Let $t_{\rm ne}$, $t_{\rm se}$, $t_{\rm sw}$ respectively denote the colors of $\sigma(H)^{\rm ne}$, $\sigma(H)^{\rm se}$, $\sigma(H)^{\rm sw}$, then
\begin{align*}
\frac{z_v(p/q)}{z_h(p/q)}&=\frac{z_v(\sigma(H))}{z_h(\sigma(H))}\circ_{\mathsf{t}}\frac{z_v(k_s)}{z_h(k_s)}
=\frac{z_h(H)}{z_v(H)}\circ_{\mathsf{t}}z_v(k_s)  \\
&=\frac{\eta_{s-2}}{\eta_{s-1}}+\frac{1-t_{\rm sw}}{1-t_{\rm se}^{-1}}b_s
+\frac{t_{\rm ne}^{-1}-t_{\rm se}}{1-t_{\rm se}}\frac{\eta_{s-2}}{\eta_{s-1}}b_s.
\end{align*}
Noticing $t_{\rm ne}=v_{s-2}$, $t_{\rm se}=u_{s-1}$, $t_{\rm sw}=v_{s-1}$, we obtain
$$z_v(p/q)=\kappa\left(\eta_{s-2}+\frac{1-v_{s-1}^{-1}}{1-u_{s-1}^{-1}}\eta_{s-1}b_s+\frac{v_{s-2}^{-1}-u_{s-1}}{1-u_{s-1}}\eta_{s-2}b_s\right)
=\kappa\eta_s.$$
Thus, $z([p/q])=[\eta_s:\eta_{s-1}]$, i.e. the assertion holds for $s=n$.

The proof is complete.
\end{proof}

\begin{exmp}\label{exmp:tangle}
\rm Let $T_1=[[2h_1],[2h_2]]$, $T_2=[[2h_1-1],[2h_2]]$. See Figure \ref{fig:doubletwist}.

For $T_1$, we have $u_1=t_2^{-1}$, $v_1=t_1$, $u_2=v_0=v_2=t_2$, so
\begin{align*}
\eta_1&=b_1=t_1(t_2-1)[h_1]_{t_1t_2},  \qquad  b_2=h_2(1-t_2),  \\
\eta_2&=1+\frac{1-t_1^{-1}}{1-t_2}b_1b_2=1+h_2(t_1-1)(t_2-1)[h_1]_{t_1t_2}.
\end{align*}

\begin{figure}[H]
  \centering
  \includegraphics[width=8.5cm]{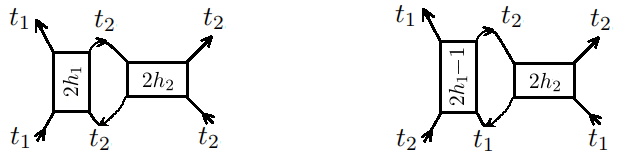}\\
  \caption{Left: $[[2h_1],[2h_2]]$. Right: $[[2h_1-1],[2h_2]]$.}\label{fig:doubletwist}
\end{figure}

For $T_2$, we have $u_1=t_1^{-1}$, $u_2=t_1$, $v_0=v_1=v_2=t_2$, so
\begin{align*}
\eta_1&=b_1=(1-t_2)[h_1]_{t_1t_2}-1, \qquad  b_2=t_2(t_1^{-1}-1)[h_2]_{t_1/t_2},   \\
\eta_2&=1+\frac{1-t_2^{-1}}{1-t_1}b_1b_2+\frac{t_2^{-1}-t_1^{-1}}{1-t_1^{-1}}b_2   \\
&=1+t_1^{-1}\big(1-t_1-(t_2-1)^2[h_1]_{t_1t_2}\big)[h_2]_{t_1/t_2}.
\end{align*}
\end{exmp}

\section{Some computations}\label{sec:computation}

\subsection{An illustrating example}

\begin{figure}[H]
  \centering
  \includegraphics[width=4cm]{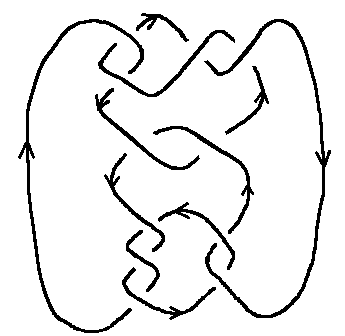}\\
  \caption{The knot $K$, with an orientation chosen.}\label{fig:example}
\end{figure}

Consider the knot $K=D(T_1\ast T_2\ast T_3)$, with $T_1=[[2],[-2]]$, $T_2=[2]$, $T_3=[1/3]+[1/2]$, as shown in Figure \ref{fig:example}.

For $T_1$, by Example \ref{exmp:tangle}, with $t_1=t^{-1}$, $t_2=t$, $h_1=1$, $h_2=-1$, we have $\eta_1=1-t^{-1}$, $\eta_2=t+t^{-1}-1$,
so by Lemma \ref{lem:rational},
$$z(T_1)=[t+t^{-1}-1:1-t^{-1}].$$

By (R1), $z(T_2)=[1-t:1]$.

By (R1) again,
$$z([1/3])=[1:t^{-2}-t^{-1}-t^{-3}], \qquad  z([1/2])=[1:t^2-t].$$
We may set $z_v(1/3)=z_v(1/2)=1$, and
$$z_h(1/3)=t^{-2}-t^{-1}-t^{-3}, \qquad  z_h(1/2)=t^2-t.$$
Note that $\Phi([1/3]^{\rm ne})=t^{-1}$, $\Phi([1/3]^{\rm se})=\Phi([1/3]^{\rm sw})=t$; with $\mathsf{t}=(t^{-1},t,t)$,
\begin{align*}
z_v(T_3)&\doteq z_h(T_3)\left(\frac{1}{z_h(1/3)}\circ_{\mathsf{t}}\frac{1}{z_h(1/2)}\right)  \\
&=z_h(1/3)z_h(1/2)\left(\frac{1}{t^{-2}-t^{-1}-t^{-3}}-t\cdot\frac{1}{t^2-t}\right)  \\
&=t^2-t+1-t^{-1}+t^{-2}.
\end{align*}

For $S\in\{T_1,T_1\ast T_2\}$, as is easy to see, $\Phi(S^{\rm ne})=t$, $\Phi(S^{\rm se})=t^{-1}$, $\Phi(S^{\rm sw})=t$.
Hence with $\mathsf{t}=\mathsf{t}'=(t,t^{-1},t)$, we may compute
\begin{align*}
\frac{z_h(T_1\ast T_2\ast T_3)}{z_v(T_1)z_v(T_2)z_v(T_3)}
&=\left(\frac{z_h(T_1)}{z_v(T_1)}\ast_{\mathsf{t}}\frac{z_h(T_2)}{z_v(T_2)}\right)\ast_{\mathsf{t}'}\frac{z_h(T_3)}{z_v(T_3)}   \\
&=\frac{t-1}{t^2-t+1}+\frac{1}{1-t}+\frac{(t-1)(t-t^2-1)}{t^4-t^3+t^2-t+1} \\
&=\frac{-t^3(t^3-3t^2+7t-9+7t^{-1}-3t^{-2}+t^{-3})}{(t^4-t^3+t^2-t+1)(t-1)(t^2-t+1)}.
\end{align*}
Thus,
\begin{align*}
\Delta_K\doteq z_h(T_1\ast T_2\ast T_3)\doteq t^3-3t^2+7t-9+7t^{-1}-3t^{-2}+t^{-3}.
\end{align*}

\subsection{A family of non-Montesinos links}

Consider the link of the form $L=D(T)$, $T=(T_1+T_2)\ast T_0\ast(T_3+T_4)$, with
$$T_0=[2h], \qquad  T_1=\Big[\frac{1}{n_1}\Big],  \qquad T_2=\Big[\frac{1}{n_2}\Big],
\qquad T_3=\Big[\frac{-1}{n_1}\Big], \qquad  T_4=\Big[\frac{-1}{n_2}\Big].$$

\begin{figure}[h]
  \centering
  \includegraphics[width=13cm]{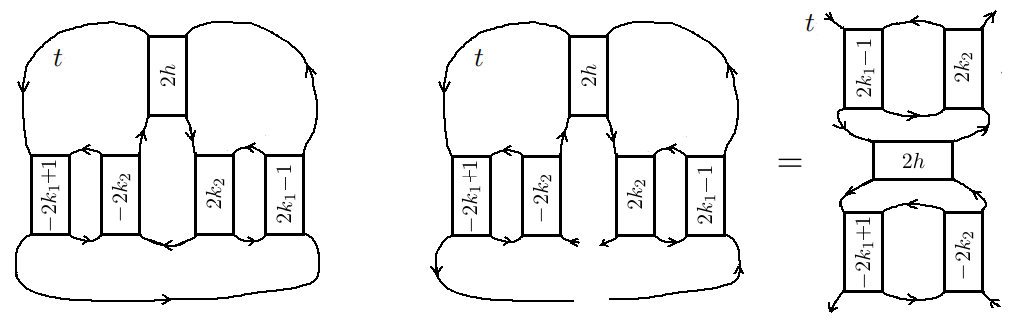}\\
  \caption{Left: the knot $L=D(T)$. Right: $T=(T_1+T_2)\ast T_0\ast(T_3+T_4)$.}\label{fig:generalized-1}
\end{figure}

Suppose $n_1=2k_1-1$, $n_2=2k_2$. Then $L$ is a knot, as shown in Figure \ref{fig:generalized-1}.
We may set $z_v(T_1)=\cdots=z_v(T_4)=z_h(T_0)=1$. By (R1),
\begin{alignat*}{2}
z_h(T_1)&=-\frac{t^{-n_1}+1}{t+1}, \qquad &z_h(T_2)&=\frac{t(t^{n_2}-1)}{t+1}, \\
z_h(T_3)&=-\frac{t^{n_1}+1}{t+1}, \qquad &z_h(T_4)&=\frac{t(t^{-n_2}-1)}{t+1}.
\end{alignat*}
Moreover, $z_v(T_0)=h(1-t)$.
Hence
\begin{align*}
z_v(T_1+T_2)&=z_h(T_1)z_h(T_2)\left(\frac{1}{z_h(T_1)}-t\cdot\frac{1}{z_h(T_2)}\right)=\frac{t(t^{-n_1}+t^{n_2})}{t+1}, \\
z_v(T_3+T_4)&=z_h(T_3)z_h(T_4)\left(\frac{1}{z_h(T_3)}-t\cdot\frac{1}{z_h(T_4)}\right)=\frac{t(t^{n_1}+t^{-n_2})}{t+1}.
\end{align*}
Consequently,
\begin{align*}
\frac{z_h(T)}{z_v(T)}
&=\frac{z_h(T_1)z_h(T_2)}{z_v(T_1+T_2)}+\frac{z_h(T_0)}{z_v(T_0)}+\frac{z_h(T_3)z_h(T_4)}{z_v(T_3+T_4)}  \\
&=\frac{(t^{-n_1}+1)(1-t^{n_2})}{(t+1)(t^{-n_1}+t^{n_2})}+\frac{1}{h(1-t)}+\frac{(t^{n_1}+1)(1-t^{-n_2})}{(t+1)(t^{n_1}+t^{-n_2})}
=\frac{1}{h(1-t)}.
\end{align*}
Since
\begin{align*}
z_v(T)=z_v(T_1+T_2)z_v(T_3+T_4)z_v(T_0)=\frac{t^2(t^{-n_1}+t^{n_2})(t^{n_1}+t^{-n_2})}{(t+1)^2}\cdot h(1-t),
\end{align*}
we have
$$\Delta_L\doteq\frac{t^{n_1+n_2}+t^{-n_1-n_2}+2}{t+t^{-1}+2}.$$

In particular, if $n_1+n_2\in\{\pm1\}$, in which case $L$ is called a {\it generalized Kinoshita-Terasaka knot}
(see \cite[Page 84]{Li97}), then $\Delta_L\doteq 1$.

\begin{figure}[H]
  \centering
  \includegraphics[width=13cm]{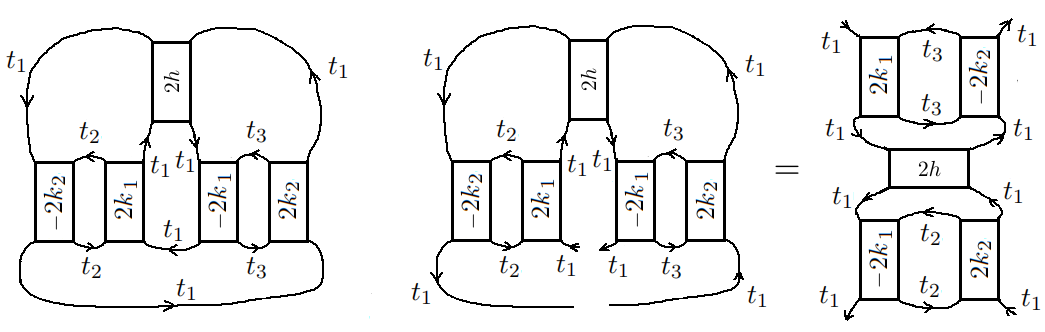}\\
  \caption{Left: the link $L=D(T)$. Right: $T=(T_1+T_2)\ast T_0\ast(T_3+T_4)$.}\label{fig:generalized-2}
\end{figure}

Now suppose $n_1=2k_1$, $n_2=-2k_2$. Then $L$ is a $3$-component link, as shown in Figure \ref{fig:generalized-2}.
We may set $z_v(T_1)=\cdots=z_v(T_4)=z_h(T_0)=1$. Then
\begin{alignat*}{2}
z_h(T_1)&=(t_3-1)[-k_1]_{t_1t_3},   \qquad   &z_h(T_2)&=t_3(t_1-1)[-k_2]_{t_1t_3}, \\
z_h(T_3)&=(t_2-1)[k_1]_{t_1t_2},    \qquad   &z_h(T_4)&=t_2(t_1-1)[k_2]_{t_1t_2},
\end{alignat*}
and $z_v(T_0)=h(1-t_1)$.
Hence
\begin{align*}
z_v(T_1+T_2)&=z_h(T_1)z_h(T_2)\left(\frac{1}{z_h(T_1)}+\frac{1-t_1}{1-t_3^{-1}}\cdot\frac{1}{z_h(T_2)}\right)  \\
&=t_3(t_1-1)([-k_2]_{t_1t_3}-[-k_1]_{t_1t_3}),  \\
z_v(T_3+T_4)&=z_h(T_3)z_h(T_4)\left(\frac{1}{z_h(T_3)}+\frac{1-t_1}{1-t_2^{-1}}\cdot\frac{1}{z_h(T_4)}\right)  \\
&=t_2(t_1-1)([k_2]_{t_1t_2}-[k_1]_{t_1t_2}).
\end{align*}

Consequently, when $k_1\ne k_2$,
\begin{align*}
\frac{z_h(T)}{z_v(T)}&=\frac{z_h(T_1)z_h(T_2)}{z_v(T_1+T_2)}+\frac{z_h(T_0)}{z_v(T_0)}+\frac{z_h(T_3)z_h(T_4)}{z_v(T_3+T_4)}   \\
&=\frac{(t_3-1)[-k_1]_{t_1t_3}[-k_2]_{t_1t_3}}{[-k_2]_{t_1t_3}-[-k_1]_{t_1t_3}}+\frac{1}{h(1-t_1)}
+\frac{(t_2-1)[k_1]_{t_1t_2}[k_2]_{t_1t_2}}{[k_2]_{t_1t_2}-[k_1]_{t_1t_2}},
\end{align*}
implying
\begin{align*}
\Delta_L&\doteq\frac{1}{1-t_1}\frac{z_h(T)}{z_v(T)}z_v(T_1+T_2)z_v(T_3+T_4)z_v(T_0)  \\
&\doteq ht_2t_3(t_1-1)^2\Big((t_3-1)[-k_1]_{t_1t_3}[-k_2]_{t_1t_3}\big([k_2]_{t_1t_2}-[k_1]_{t_1t_2}\big) \\
&\hspace{30mm}+(t_2-1)[k_1]_{t_1t_2}[k_2]_{t_1t_2}\big([-k_2]_{t_1t_3}-[-k_1]_{t_1t_3}\big)\Big) \\
&\ \ \ \ +t_2t_3(1-t_1)\big([-k_2]_{t_1t_3}-[-k_1]_{t_1t_3}\big)\big([k_2]_{t_1t_2}-[k_1]_{t_1t_2}\big).
\end{align*}
When $k_1=k_2$, $z_v(T_1+T_2)=z_v(T_3+T_4)=0$; applying a formula in (R2) twice, we easily find $z_h(T)=0$. so that $\Delta_L=0$.

\begin{rmk}
\rm In this way, we may construct a lot of links $L$ with $\Delta_L=0$.
\end{rmk}

\subsection{Montesinos links}

Each rational tangle $T=[p/q]$ is made of two curves. There are three cases: (1) $p,q$ are both odd; (2) $p$ is odd and $q$ is even; (3) $p$ is even and $q$ is odd; respectively, we say that $T$ has type 1, 2, 3. The way in which the ends of $T$ are connected to each other depends on the type, as shown in Figure \ref{fig:type}.

\begin{figure}[H]
  \centering
  \includegraphics[width=9cm]{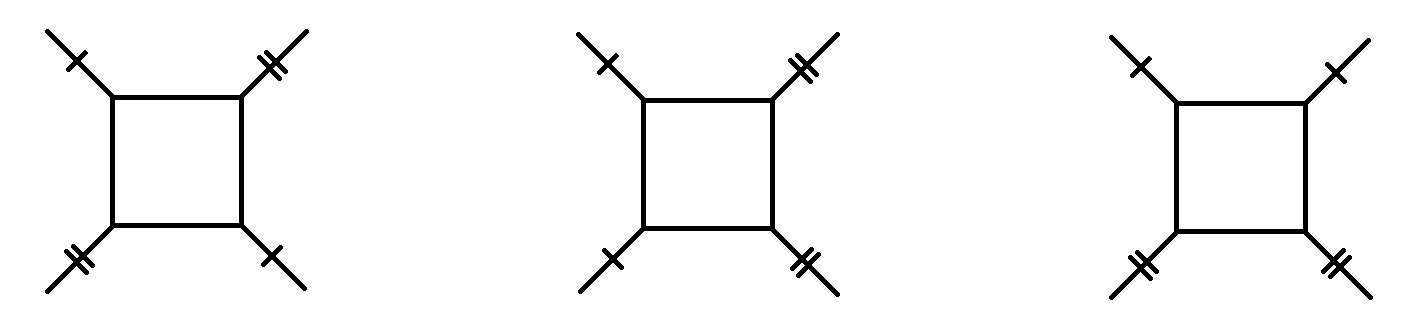}\\
  \caption{Left: if $T$ has type 1, then $T^{\rm nw}$, $T^{\rm se}$ belong to the same curve, as indicated by single-strand line segments, and $T^{\rm ne}$, $T^{\rm sw}$ belong to the other, as indicated by double-strand line segments. Middle: the situation when $T$ has type 2. Right: the situation when $T$ has type 3.}\label{fig:type}
\end{figure}

\begin{figure}[h]
  \centering
  \includegraphics[width=12cm]{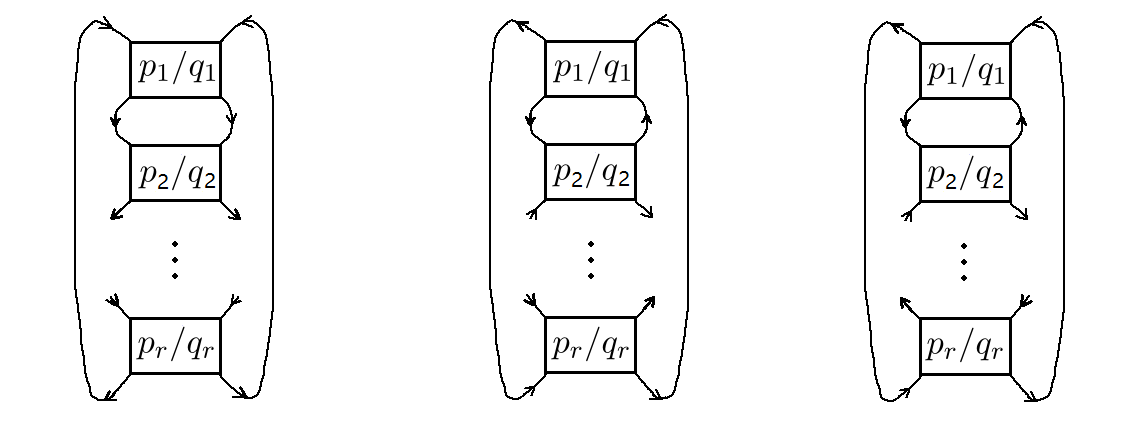}\\
  \caption{Left: $L$ is odd. Middle: $L$ is even, with $n_1(L)$ odd. Right: $L$ is even, with $n_1(L)$ even.
  Here type 2 tangles are not explicitly drawn out.}\label{fig:Montesinos}
\end{figure}

Let $L=D([p_1/q_1]\ast\cdots\ast[p_r/q_r])$ be a Montesinos link, with $r\ge 3$. Let $n_1(L),n_3(L)$ respectively denote the number of tangles of type 1 and 3. Then $L$ is a knot if and only if one of the following occurs.
\begin{enumerate}
  \item If $n_3(L)=0$ and $n_1(L)$ is odd, then $L$ is called odd. No matter the $q_i$'s are odd or even,
        the tangles can be oriented in the way shown in the left part of Figure \ref{fig:Montesinos}.
  \item If $n_3(L)=1$, then $L$ is called even. In this case, we can assume that $2\mid p_r$ and $2\nmid p_i$ for $1\le i<r$.
        When $n_1(L)$ is odd (resp. even), $L$ can be oriented as in the middle (resp. right) part of Figure \ref{fig:Montesinos}.
\end{enumerate}

\begin{thm}\label{thm:Montesinos-knot}
Let $L$ be the Montesinos knot $D([p_1/q_1]\ast\cdots\ast[p_r/q_r])$.
\begin{enumerate}
  \item[\rm(a)] When $L$ is odd, i.e. $n_3(L)=0$ and $n_1(L)$ is odd,
        $$\Delta_L\doteq\frac{1}{t+1}
        \left(\prod_{i=1}^r\big((t+1)z_h(p_i/q_i)+z_v(p_i/q_i)\big)-\prod_{i=1}^rz_v(p_i/q_i)\right).$$
  \item[\rm(b)] When $L$ is even, i.e. $n_3(L)=1$, set $\epsilon_i=(-1)^{q_1+\cdots+q_{i-1}}$, then
        $$\Delta_L\doteq (1-t)\left(\prod_{i=1}^rz_v(p_i/q_i)\right)\cdot\sum_{i=1}^{r}\frac{z_h(p_i/q_i)}{(1-t^{\epsilon_i})z_v(p_i/q_i)}.$$
\end{enumerate}
\end{thm}

\begin{proof}
For $1\le i\le r$, let
$$S_i=[p_1/q_1]\ast\cdots\ast[p_i/q_i]),  \qquad  g_i=\frac{z_h(p_i/q_i)}{z_v(p_i/q_i)},     \qquad      u_i=\frac{z_h(S_i)}{z_v(S_i)}.$$

(a) For each $2\le i\le r$, when vertically composing $S_{i-1}$ with $[p_i/q_i]$, we have
$\Phi(S_{i-1}^{\rm ne})=t^{-1}$, $\Phi(S_{i-1}^{\rm se})=\Phi(S_{i-1}^{\rm sw})=t$, hence by (\ref{eq:fraction-v}),
$$u_{i}=u_{i-1}+g_i+(t+1)u_{i-1}g_i.$$
Consequently,
$$u_r=\frac{1}{t+1}\prod_{i=1}^r\big((t+1)g_i+1\big)-\frac{1}{t+1}.$$
Then the formula for $\Delta_L$ follows by Theorem \ref{thm:main}.

(b) For each $2\le i\le r$, when vertically composing $S_{i-1}$ with $[p_i/q_i]$, we have
$\Phi(S_{i-1}^{\rm ne})=t^{-1}$, $\Phi(S_{i-1}^{\rm se})=t^{\varepsilon_i}$, $\Phi(S_{i-1}^{\rm sw})=t^{-\varepsilon_i}$,
for some $\varepsilon_i\in\{\pm1\}$, hence by (\ref{eq:fraction-v}),
$$u_{i}=u_{i-1}+\frac{1-t}{1-t^{\varepsilon_i}}g_i.$$
Set $\varepsilon_1=1$. To determine $\varepsilon_i$ for $2\le i\le r$, it suffices to know the color of
$[p_{i-1}/q_{i-1}]^{\rm se}=([p_i/q_i]^{\rm ne})^{-1}$ specified by the orientation given in Figure \ref{fig:Montesinos}. Observe that if $q_{i-1}$ is odd (resp. even), then $\Phi([p_i/q_i]^{\rm ne})$ is inverse to (resp. the same as) $\Phi([p_{i-1}/q_{i-1}]^{\rm ne})$, so that $\varepsilon_{i}=-\varepsilon_{i-1}$ (resp. $\varepsilon_{i}=\varepsilon_{i-1}$). Consequently, $\varepsilon_i=(-1)^{q_1+\cdots+q_{i-1}}=\epsilon_i$.
Thus,
$$u_r=\sum_{i=1}^{r}\frac{1-t}{1-t^{\epsilon_i}}g_i,$$
and the formula for $\Delta_L$ follows.
\end{proof}

As a corollary, we recover \cite[Theorem 1]{Be25}:
\begin{cor}\label{cor:pretzel-knot}
Let $L$ be the pretzel knot $D([p_1]\ast\cdots\ast[p_r])$.
\begin{enumerate}
  \item[\rm(a)] When $r$ and all the $p_i$'s are odd,
        $$\Delta_L\doteq\frac{1}{2^{r-1}}\sum_{k=0}^{(r-1)/2}\sigma_{2k}(p_1,\ldots,p_r)(t+1)^{r-1-2k}(t-1)^{2k},$$
        where $\sigma_{2k}(p_1,\ldots,p_r)$ is the $2k$-th elementary polynomial in $p_1,\ldots,p_r$.
  \item[\rm(b)] When $r$ is even, $p_r=2h$, and $p_i$ is odd for $1\le i<r$,
        $$\Delta_L\doteq \left(\prod_{i=1}^{r-1}\frac{1+t^{p_i}}{1+t}\right)\left(t^{h}+(t^{h}-t^{-h})
        \left(\sum_{i=1}^{r-1}\frac{t^{p_i}}{1+t^{p_i}}-\frac{r}{2}\right)\right).$$
  \item[\rm(c)] When $r$ is odd, $p_r=2h$, and $p_i$ is odd for $1\le i<r$,
        $$\Delta_L\doteq \left(\prod_{i=1}^{r-1}\frac{1+t^{p_i}}{1+t}\right)\left(1+h(t^{-1}-t)
        \left(\sum_{i=1}^{r-1}\frac{t^{p_i}}{1+t^{p_i}}-\frac{r-1}{2}\right)\right).$$
\end{enumerate}
\end{cor}

\begin{proof}
Recall Convention \ref{conv:zh} that $z_h(p_i)=1$ for each $i$.

(a) Suppose $p_i=2h_i-1$. Then $z_v(p_i)=h_i(1-t)-1$. Hence
\begin{align*}
\Delta_L&\doteq\frac{1}{t+1}\left(\prod_{i=1}^r(t+h_i(1-t))-\prod_{i=1}^r(h_i(1-t)-1)\right)  \\
&\doteq\frac{1}{2^r(t+1)}\left(\prod_{i=1}^r(t+1-p_i(t-1))+\prod_{i=1}^r(t+1+p_i(t-1))\right)  \\
&=\frac{1}{2^{r-1}}\sum_{k=0}^{(r-1)/2}\sigma_{2k}(p_1,\ldots,p_r)(t+1)^{r-1-2k}(t-1)^{2k}.
\end{align*}

(b) Note that $\epsilon_i=(-1)^{i-1}$, and
\begin{align*}
z_v(p_i)&=(1-t^{-\epsilon_i})\Big[\frac{p_i+1}{2}\Big]_{t^{-2\epsilon_i}}-1=-\frac{1+t^{-p_i\epsilon_i}}{1+t^{\epsilon_i}}, \quad 1\le i<r,  \\
z_v(p_r)&=t(t-1)[h]_{t^2}=\frac{t(t^{2h}-1)}{1+t}.
\end{align*}
Hence for $1\le i<r$,
$$\frac{z_h(p_i)}{(1-t^{\epsilon_i})z_v(p_i)}=\frac{t+1}{t-1}\left(\frac{t^{p_i}}{1+t^{p_i}}-\frac{(-1)^i+1}{2}\right).$$
Then the formula follows.

(c) In this case, $z_v(p_r)=h(1-t^{-1})$, and the remaining steps are similar as (b).
\end{proof}

\medskip

Now we turn to Montesinos links with at least two components.

There are two possibilities:
\begin{enumerate}
  \item[\rm(a)] If $n_3(L)=0$ and $n_1(L)$ is even, then $L$ has 2 components. We orient $L$ in such a way that
        $\Phi([p_1/q_1]^{\rm ne})=t_1^{-1}$, and $\Phi([p_1/q_1]^{\rm se})=t_2.$
  \item[\rm(b)] If $n:=n_3(L)\ge 2$, then $L$ has $n$ components. Each type 3 tangle is made of two curves belonging to different components.
        Suppose $2\mid p_i$ exactly for $i=r_1,\ldots,r_n$, with $1\le r_1<\cdots<r_n=r$. Set $r_0=0$.
        Numerate the components of $L$ so that the $k$-th one is made of the curve of $[p_{r_{k-1}}/q_{r_{k-1}}]$ containing the south ends, $[p_i/q_i]$ for $r_{k-1}<i\le r_k$, and the curve of $[p_{r_k}/q_{k_k}]$ containing the north ends. Orient the $k$-th component in such a way that $\Phi([p_{r_{k-1}}/q_{r_{k-1}}]^{\rm se})=t_k$.
\end{enumerate}

\begin{thm}\label{thm:Montesinos-link}
Let $L$ be the Montesinos link $D([p_1/q_1]\ast\cdots\ast[p_r/q_r])$.
\begin{enumerate}
  \item[\rm(a)] When $n_3(L)=0$ and $n_1(L)$ is even, let $\nu_1=1$, and for $1<i\le r$, let $\nu_i=1$ (resp. $\nu_i=2$)
        if $q_1+\cdots+q_{i-1}$ is even (resp. odd). Then
        $$\Delta_L\doteq\frac{1}{t_1t_2-1}\left(\prod_{i=1}^r
        \left(\frac{t_1t_2-1}{t_{\nu_i}-1}z_h(p_i/q_i)+z_v(p_i/q_i)\right)-\prod_{i=1}^rz_v(p_i/q_i)\right).$$
  \item[\rm(b)] Suppose $n=n_3(L)\ge 2$. For $i$ with $r_{k-1}<i\le r_k$, set $\nu_i=k$ and $\epsilon_i=(-1)^{q_{r_{k-1}+1}+\cdots+q_{i-1}}$.
        Then
        \begin{align*}
        \Delta_L\doteq\prod_{i=1}^rz_v(p_i/q_i)\cdot\sum_{i=1}^r\frac{z_h(p_i/q_i)}{(1-t_{\nu_i}^{\epsilon_i})z_v(p_i/q_i)}.
        \end{align*}
\end{enumerate}
\end{thm}

\begin{proof}
(a) Adopt the notations $g_i,u_i$ introduced in the proof of Theorem \ref{thm:Montesinos-knot}. This time,
$$u_i=u_{i-1}+\frac{1-t_1}{1-t_{\nu_i}}g_i+\frac{t_1t_2-1}{t_{\nu_i}-1}u_{i-1}g_i.$$
Hence
$$u_r=\frac{t_1-1}{t_1t_2-1}\left(\prod_{i=1}^r\left(\frac{t_1t_2-1}{t_{\nu_i}-1}g_i+1\right)-1\right).$$
Then the formula follows.

(b) Similarly as the proof of Theorem \ref{thm:Montesinos-knot} (b).
\end{proof}

Specializing to pretzel links, we obtain explicit formulas:
\begin{cor}\label{cor:pretzel-link}
Let $L$ be the pretzel link $D([p_1]\ast\cdots\ast[p_r])$.
\begin{enumerate}
  \item[\rm(a)] When $r$ is even and $p_i=2h_i-1$ for $1\le i\le r$,
        \begin{align*}
        \Delta_L\doteq \frac{1}{t_1t_2-1}
        \Bigg(&\prod_{k=1}^{r/2}((t_1-1)[h_{2k-1}]_{t_2/t_1}-t_1)((t_2-1)[h_{2k}]_{t_1/t_2}-t_2)    \\
        &-\prod_{k=1}^{r/2}((t_2-1)[h_{2k-1}]_{t_2/t_1}+1)((t_1-1)[h_{2k}]_{t_1/t_2}+1)\Bigg).
        \end{align*}
  \item[\rm(b)] When $n=n_3(L)\ge 2$ and $p_{r_k}=2\ell_k$ for $1\le k\le n$ with $1\le r_1<\cdots<r_n=r$, and all the other $p_i$'s are odd,
        \begin{align*}
        \Delta_L\doteq&\left(\prod_{k=1}^n\mathfrak{g}_k\cdot\prod_{i=r_{k-1}+1}^{r_k}\frac{t_k^{p_i}+1}{t_k+1}\right)  \\
        &\cdot\sum_{k=1}^n\left(\frac{1}{(1-t_k^{\nu_k})\mathfrak{g}_k}-\Big\lfloor\frac{r_k-r_{k-1}}{2}\Big\rfloor
        +\sum_{i=r_{k-1}+1}^{r_k}\frac{(1+t_k)t_k^{p_i}}{(1-t_k)(1+t_k^{p_i})}\right),
        \end{align*}
        where $\nu_k=(-1)^{r_k-r_{k-1}}$ and $\mathfrak{g}_k=(t_{k+1}-1)[\ell_k]_{t_k^{\nu_k}t_{k+1}}$.
\end{enumerate}
\end{cor}

\begin{proof}
(a) We have
$$\nu_i=\begin{cases} 1,&2\nmid i \\ 2,&2\mid i \end{cases},   \qquad
z_v(p_i)=\begin{cases} (1-t_2)[h_i]_{t_2/t_1}-1, &2\nmid i \\   (1-t_1)[h_i]_{t_1/t_2}-1, &2\mid i  \end{cases}.$$
Then the formula follows from Theorem \ref{thm:Montesinos-link} (a).

(b) For $r_{k-1}<i\le r_k$, we have $\epsilon_i=(-1)^{i-1-r_{k-1}}$; for the tangle $[p_i]$,
\begin{align*}
t_{\rm ne}=t_k^{(-1)^{i-r_{k-1}}}=t_k^{-\epsilon_i}, \qquad
t_{\rm se}=\begin{cases} t_k^{-\epsilon_i}, &r_{k-1}<i<r_k \\ t_{k+1},&i=r_k  \end{cases}.
\end{align*}
Hence
\begin{align*}
z_v(p_i)&=-\frac{t_k^{-p_i\epsilon_i}+1}{t_k^{\epsilon_i}+1}, \quad  r_{k-1}<i<r_k;   \\
z_v(p_{r_k})&=t_k^{\nu_k}(t_{k+1}-1)[\ell_k]_{t_k^{\nu_k}t_{k+1}}.
\end{align*}
Then the formula follows from Theorem \ref{thm:Montesinos-link} (b).
\end{proof}

\section{Proof of Theorem \ref{thm:main}}

\subsection{Set up}

Let $L=D(T)$, with $T\in\mathcal{T}_{\rm ar}$. Recall the notations used in Section 1.

For $\mathfrak{c}_i=(x_j,x_k,x_\ell)$, by which we mean the crossing made of the arcs $x_j$, $x_k$, $x_\ell$, as shown in Figure \ref{fig:pm},
denote $j,k,\ell$ by $\overline{i},\underline{i},i'$, respectively.
\begin{figure}[H]
  \centering
  \includegraphics[width=8cm]{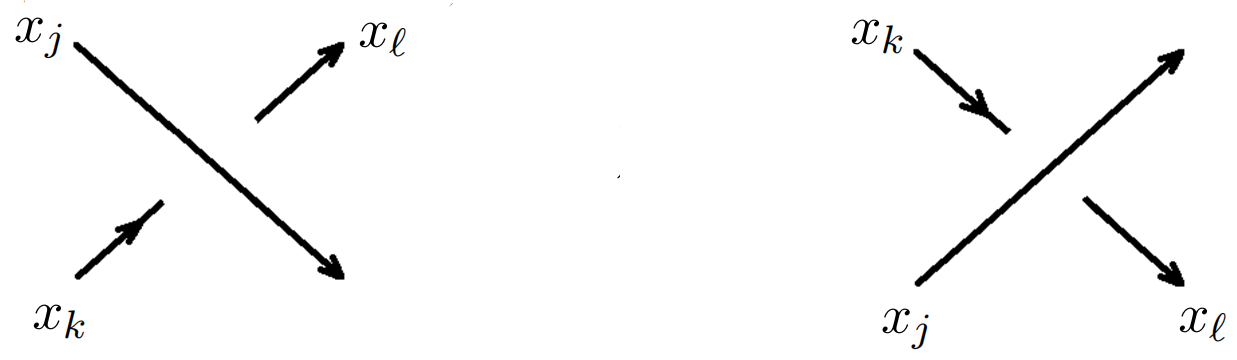}\\
  \caption{Left: a positive crossing. Right: a negative crossing.} \label{fig:pm}
\end{figure}

To simplify the notation, let $\tau_j=t_{\nu(j)}$ for each $j$.

When $\mathfrak{c}_i$ is positive, the corresponding relator is $r_i=x_jx_kx_j^{-1}x_{\ell}^{-1}$,
$$\frac{\partial r_i}{\partial x_j}=1-x_jx_kx_j^{-1}, \qquad \frac{\partial r_i}{\partial x_k}=x_j, \qquad
\frac{\partial r_i}{\partial x_\ell}=-x_jx_kx_j^{-1}x_{\ell}^{-1},$$
so
$$\Phi\Big(\frac{\partial r_i}{\partial x_j}\Big)=1-\tau_k, \qquad  \Phi\Big(\frac{\partial r_i}{\partial x_k}\Big)=\tau_j,
\qquad \Phi\Big(\frac{\partial r_i}{\partial x_\ell}\Big)=-1;$$
when $\mathfrak{c}_i$ is negative, the corresponding relator is $r_i=x_j^{-1}x_kx_jx_\ell^{-1}$,
$$\frac{\partial r_i}{\partial x_j}=-x_j^{-1}+x_j^{-1}x_k, \qquad \frac{\partial r_i}{\partial x_k}=x_j^{-1}, \qquad
\frac{\partial r_i}{\partial x_\ell}=-x_j^{-1}x_kx_jx_\ell^{-1},$$
so
$$\Phi\Big(\frac{\partial r_i}{\partial x_j}\Big)=\tau_j^{-1}(\tau_k-1), \qquad  \Phi\Big(\frac{\partial r_i}{\partial x_k}\Big)=\tau_j^{-1},
\qquad \Phi\Big(\frac{\partial r_i}{\partial x_\ell}\Big)=-1.$$
In either case (remembering that $\tau_\ell=\tau_k$),
\begin{align*}
\Phi\Big(\frac{\partial r_i}{\partial x_j},\frac{\partial r_i}{\partial x_k},\frac{\partial r_i}{\partial x_\ell}\Big)
\left(\begin{array}{ccc} 1-\tau_j & \ & \  \\ \ & 1-\tau_k & \ \\ \ & \ & 1-\tau_\ell \end{array}\right)
=(1-\tau_\ell)(1-\tau_j^{\epsilon},\tau_j^{\epsilon},-1),
\end{align*}
where $\epsilon=1$ (resp. $\epsilon=-1$) if $\mathfrak{c}_i$ is positive (resp. negative).

Thus, similarly as in the beginning of \cite[Section 3]{Ch21}, $MD=D'Q_L$, where
\begin{itemize}
  \item $D$ is the $n\times n$ diagonal matrix whose $j$-th diagonal entry is $1-\tau_{j}$,
  \item $D'$ is the $n\times n$ diagonal matrix whose $i$-th entry is $1-\tau_{i'}$,
  \item $Q_L=(q_{ij})_{n\times n}$, with $q_{i,j}=0$ for $j\notin\{\overline{i},\underline{i},i'\}$, and
        \begin{align*}
        q_{i,\overline{i}}=1-\tau_{\overline{i}}^{\epsilon}, \qquad  q_{i,\underline{i}}=\tau_{\overline{i}}^{\epsilon},
        \qquad  q_{i,i'}=-1,
        \end{align*}
        where $\epsilon=1$ (resp. $\epsilon=-1$) if $\mathfrak{c}_i$ is positive (resp. negative).
\end{itemize}

\begin{rmk}
\rm A crossing $\mathfrak{c}_i$ may give rise to several different relators which are conjugate to each other, and $\partial r_i/\partial x_j$, $\partial r_i/\partial x_k$, $\partial r_i/\partial x_\ell$ depend on the choice of $r_i$. However, up to a factor of the form $\pm\tau_{\overline{i}}$ or $\pm\tau_{\overline{i}}^{-1}$, the $i$-th row of $Q_L$ only depends on $\mathfrak{c}_i$.
This is why in Section 2, we are content with defining $z_v(T),z_h(T)$ up to the same unit.
\end{rmk}

For any $i,j$, let $M_{\neg i}^{\neg j}$ denote the matrix obtained from $M$ by deleting the $i$-th row and the $j$-th column.
Then $M_{\neg i}^{\neg j}D_{\neg j}^{\neg j}=(D')^{\neg i}_{\neg i}(Q_L)_{\neg i}^{\neg j}$, so
\begin{align}
\frac{\det\big(M_{\neg i}^{\neg j}\big)}{1-\tau_j}\doteq\frac{\det\big((Q_L)_{\neg i}^{\neg j}\big)}{1-\tau_{i'}}.  \label{eq:AP}
\end{align}

Suppose $T^{\rm ne}$ or $(T^{\rm ne})^{-1}$ is one of the two lower arcs constituting the crossing $\mathfrak{c}_{i_0}$, i.e.,
$T^{\rm ne}\in\{x_{\underline{i_0}}^{\pm1},x_{i'_0}^{\pm1}\}$, then $\tau_{i'_0}\in\{\Phi(T^{\rm ne})^{\pm1}\}$.

Let $\Lambda$ denote the free abelian group generated by $t_1,\ldots,t_m$, whose elements are denoted multiplicatively.
Following \cite{Li97}, construct a CW complex $P$ by taking one 0-cell $O$, oriented 1-cells $A_1,\ldots,A_n$ with $A_j$ identified with $x_j$, and oriented 2-cells $B_1,\ldots,B_n$ such that $\partial B_i$ is attached according to $r_i$. Let $\widetilde{P}$ denote the cover of $P$ corresponding to the kernel of the map $\pi(L)\to\Lambda$ determined by $x_j\mapsto t_{\nu(j)}=\tau_j$. Choose a lift $\widetilde{O}$ for $O$, let $\widetilde{A}_j$ be the lift of $A_j$ that starts at $\widetilde{O}$, and let $\widetilde{B}_i$ be the lift of $B_i$ such that $\partial\widetilde{B}_i$ is the lift of $r_i$ starting at $\widetilde{V}$.
Set up the chain complex of $\mathbb{Z}[\Lambda]$-modules
$$C_2(\widetilde{P})\stackrel{d_2}\longrightarrow C_1(\widetilde{P})\stackrel{d_1}\longrightarrow C_0(\widetilde{P}),$$
where $C_i(\widetilde{P})$ is freely generated by the specified $i$-cells, and
$$d_2(\widetilde{B}_i)={\sum}_j\Phi\Big(\frac{\partial r_i}{\partial x_j}\Big)\cdot \widetilde{A}_j.$$
Then $M$ is the presentation matrix for $\mathcal{N}:=C_1(\widetilde{P})/{\rm Im}(d_2)$, with respect to the generators $\check{A}_1,\ldots,\check{A}_n$, where $\check{A}_j$ is represented by $\widetilde{A}_j$.

Extending the coefficients from $\mathbb{Z}[\Lambda]$ to $\mathbb{Q}(\Lambda)$, we put
\begin{align}
\xi_{x_j}=(1-\tau_j)^{-1}\check{A}_j.   \label{eq:xi}
\end{align}
Then in $\mathcal{N}\otimes\mathbb{Q}(\Lambda)$, $M(\check{A}_1,\ldots,\check{A}_n)=0$ is equivalent to $Q_L(\xi_{x_1},\ldots,\xi_{x_n})=0$.

\begin{rmk}\label{rmk:independence}
\rm We can extend the definition of $\xi_x$ to all directed arcs $x$ by putting $\xi_x=(1-\Phi(x))^{-1}\widetilde{A}_x$, where $\widetilde{A}_x$ is the lifting of the oriented 1-cell corresponding to $x$. Note that $A_j^{-1}$ lifts to $-\tau_j^{-1}\widetilde{A}_j$, so
$$\xi_{x_j^{-1}}=(1-\tau_j^{-1})^{-1}(-\tau_j^{-1}\check{A}_j)=\xi_{x_j}.$$
This enables us to regard $\xi_x$ as an element associated to an undirected arc.
\end{rmk}

The goal is to transform $Q_L$ into a highly simplified matrix.
To this end, we generalize the formalism to tangles.

\subsection{Working with tangles}

Given oriented $(T,\alpha)\in\mathcal{T}_{\rm ar}^c$, numerate its crossings as $\mathfrak{c}_1,\ldots,\mathfrak{c}_n$, and numerate the directed arcs as $x_1,\ldots,x_{n+2}$, where the direction of $x_j$ is given by the orientation. Let $\tau_j=\alpha(x_j)$, and let $\Lambda$ denote the free abelian group generated by the $\tau_j$'s, whose elements are denoted multiplicatively.
Construct a matrix $Q_T$ by a manner similar as $Q_L$, by putting $Q_T=(q_{ij})_{n\times(n+2)}$, with $q_{i,j}=0$ for $j\notin\{\overline{i},\underline{i},i'\}$, and
\begin{align}
q_{i,\overline{i}}=1-\tau_{\overline{i}}^{\epsilon}, \qquad  q_{i,\underline{i}}=\tau_{\overline{i}}^{\epsilon},
\qquad  q_{i,i'}=-1,  \label{eq:QT}
\end{align}
where $\epsilon=1$ (resp. $\epsilon=-1$) if $\mathfrak{c}_i$ is positive (resp. negative).

Regarding $T$ as a $1$-manifold embedded in $\mathbb{R}^3$, take a $3$-ball $\mathfrak{B}$ containing $T$ with 
$T\cap\partial \mathfrak{B}=\partial T$.
Then $\pi_1(\mathfrak{B}\setminus T)\cong\langle x_1,\ldots,x_{n+2}\mid r_1,\ldots,r_n\rangle$, where $r_i$ is contributed by $\mathfrak{c}_i$.

Similarly as above, construct CW complexes $P_T$, $\widetilde{P}_T$ and the chain complex of $\mathbb{Z}[\Lambda]$-modules
$C_2(\widetilde{P}_T)\stackrel{d_2}\longrightarrow C_1(\widetilde{P}_T)\stackrel{d_1}\longrightarrow C_0(\widetilde{P}_T)$.
Note that if $w\in F_n$ represents $1$ in $\pi_1(E_T)$, then in $\mathcal{N}_T:=C_1(\widetilde{P}_T)/{\rm Im}(d_2)$,
\begin{align}
{\sum}_j\Phi\Big(\frac{\partial w}{\partial x_j}\Big)\cdot\check{A}_j=0.   \label{eq:trivial}
\end{align}
We can introduce $\xi_x$ for each directed arc $x$ similarly as (\ref{eq:xi}), and also show $\xi_x=\xi_{x^{-1}}$.
Furthermore, $Q_T(\xi_{x_1},\ldots,\xi_{x_{n+2}})=0$ in $\mathcal{N}_T\otimes\mathbb{Q}(\Lambda)$.

For a subtangle $S$ of $T$ and $\o\in\{{\rm nw},{\rm ne},{\rm sw},{\rm se}\}$, let $\xi^{\o}=\xi^{\o}_S=\xi_{S^{\o}}$.
Consider the following assertions:
\begin{enumerate}
  \item[\rm(v)] For each arc $x$, there exists $b^x\in\mathbb{Q}(\Lambda)$ with $\xi_x=(1-b^x)\xi^{\rm nw}+b^x\xi^{\rm ne}$.
  \item[\rm(h)] For each arc $x$, there exists $c^x\in\mathbb{Q}(\Lambda)$ with $\xi_x=(1-c^x)\xi^{\rm nw}+c^x\xi^{\rm sw}$.
\end{enumerate}

If $S=[1]$, then $\xi^{\rm se}=\xi^{\rm nw}$; in the sprit of Remark \ref{rmk:independence},
reversing the directions of the arcs if necessary, we may assume that the unique crossing takes the form in the left part of Figure \ref{fig:pm}, then from (\ref{eq:QT}) we can see
\begin{align}
\xi^{\rm ne}=(1-\alpha(S^{\rm se}))\xi^{\rm nw}+\alpha(S^{\rm se})\xi^{\rm sw};  \label{eq:[1]}
\end{align}
if $S=[-1]$, then $\xi^{\rm sw}=\xi^{\rm ne}$, and from (\ref{eq:QT}) we can see
\begin{align}
\xi^{\rm se}=(1-\alpha(S^{\rm sw}))\xi^{\rm nw}+\alpha(S^{\rm sw})\xi^{\rm ne}.   \label{eq:[-1]}
\end{align}
Hence the assertions (v), (h) are true when $S=[\pm1]$.

Suppose (v) holds, then $\xi^{\rm sw}=(1-b^{\rm sw})\xi^{\rm nw}+b^{\rm sw}\xi^{\rm ne}$;
when $b^{\rm sw}\ne 0$,
$$\xi^{\rm ne}=\Big(1-\frac{1}{b^{\rm sw}}\Big)\xi^{\rm nw}+\frac{1}{b^{\rm sw}}\xi^{\rm sw},$$
so we may rewrite each $\xi_x$ as
$\xi_x=(1-c^x)\xi^{\rm nw}+c^x\xi^{\rm sw}$, with
\begin{align*}
c^x=b^x/b^{\rm sw}.
\end{align*}
In particular,
\begin{align}
c^{\rm ne}=1/b^{\rm sw}, \qquad  c^{\rm se}=b^{\rm se}/b^{\rm sw}.   \label{eq:c}
\end{align}
Thus, (v) implies (h) if $b^{\rm sw}\ne 0$.
Similarly, (h) implies (v) if $c^{\rm ne}\ne 0$.

Suppose $S=S_1\ast S_2$ and the assertions (v), (h) hold for $S_1,S_2$. Since $\xi_{S_2}^{\rm nw}=\xi_{S_1}^{\rm sw}$ and
$\xi_{S_2}^{\rm ne}=\xi_{S_1}^{\rm se}$, for each arc $x$ of $S_2$, we may rewrite each
$\xi_x$ as
\begin{align*}
\xi_x&=(1-b_2^x)\big((1-b_1^{\rm sw})\xi_{S_1}^{\rm nw}+b_1^{\rm sw}\xi_{S_1}^{\rm ne}\big)
+b_2^x\big((1-b_1^{\rm se})\xi_{S_1}^{\rm nw}+b_1^{\rm se}\xi_{S_1}^{\rm ne}\big)  \\
&=(1-b^x)\xi_{S_1}^{\rm nw}+b^x\xi_{S_1}^{\rm ne}=(1-b^x)\xi_{S}^{\rm nw}+b^x\xi_{S}^{\rm ne},
\end{align*}
with $b^x=(1-b_2^x)b_1^{\rm sw}+b_2^xb_1^{\rm se}$, where the meanings of $b_1^x,b_2^x$ are self-explanatory.
Hence (v) holds for $S$, so does (h) if $b^{\rm sw}\ne 0$.

The situation is similar when $S=S_1+S_2$.

Thus, recursively we can show (v), (h) when $T$ is generic; by ``generic" we mean that $b^{\rm sw}c^{\rm ne}\ne 0$ in each step.

\bigskip

Set $b^{\rm nw}=b_T^{\rm nw}=b^{T^{\rm nw}}$, $c^{\rm nw}=c_T^{\rm nw}=c^{T^{\rm nw}}$, etc., then
\begin{alignat}{2}
\left(\begin{array}{cc} \xi^{\rm sw} \\ \xi^{\rm se} \end{array}\right)&=F_v^T
\left(\begin{array}{cc} \xi^{\rm nw} \\ \xi^{\rm ne} \end{array}\right), \qquad
&&\text{with} \quad  F_v^T=\left(\begin{array}{cc} 1-b^{\rm sw} & b^{\rm sw} \\ 1-b^{\rm se} & b^{\rm se} \end{array}\right),  \label{eq:Fv} \\
\left(\begin{array}{cc} \xi^{\rm ne} \\ \xi^{\rm se} \end{array}\right)&=F_h^T
\left(\begin{array}{cc} \xi^{\rm nw} \\ \xi^{\rm sw} \end{array}\right), \qquad
&&\text{with} \quad  F_h^T=\left(\begin{array}{cc} 1-c^{\rm ne} & c^{\rm ne} \\ 1-c^{\rm se} & c^{\rm se} \end{array}\right).  \label{eq:Fh}
\end{alignat}
In particular, from (\ref{eq:[1]}), (\ref{eq:[-1]}) we see
\begin{align}
F_h^{[1]}&=\left(\begin{array}{cc} 1-\alpha([1]^{\rm se}) & \alpha([1]^{\rm se}) \\ 1 & 0 \end{array}\right),   \label{eq:Fh(1)}   \\
F_v^{[-1]}&=\left(\begin{array}{cc} 0 & 1 \\ 1-\alpha([-1]^{\rm sw}) & \alpha([-1]^{\rm sw}) \end{array}\right).   \label{eq:Fv(-1)}
\end{align}

We may rephrase the assertions (v), (h) as: in generic case, $Q_T$ can be transformed into either of
\begin{align*}
N_v^T:&=\left(\begin{array}{cccccc} 1-b^{x_1} & b^{x_1} & -1 & \ & \ & \ \\ \vdots & \vdots & \ & \ddots & \ & \ \\
1-b^{\rm sw} & b^{\rm sw} & \ & \ & -1 & \  \\  1-b^{\rm se} & b^{\rm se} & \ & \ & \ & -1 \end{array}\right),   \\
N_h^T:&=\left(\begin{array}{cccccc} 1-c^{x_1} & c^{x_1} & -1 & \ & \ & \ \\ \vdots & \vdots & \ & \ddots & \ & \ \\
1-c^{\rm ne} & c^{\rm ne} & \ & \ & -1 & \  \\  1-c^{\rm se} & c^{\rm se} & \ & \ & \ & -1 \end{array}\right),
\end{align*}
i.e., there exist invertible matrices $U_v^T$, $U_h^T$ over $\mathbb{Q}(\Lambda)$ and permutation matrices $P_v^T$, $P_h^T$ such that $U_v^TQ_TP_v^T=N_v^T$, and $U_h^TQ_TP_h^T=N_h^T$.

As is easy to see, $U_v^T$, $U_h^T$ are unique.
Let
$$\tilde{z}_v(T)=1/\det(U_v^T), \qquad  \tilde{z}_h(T)=1/\det(U_h^T).$$
In particular, by (\ref{eq:Fh(1)}), (\ref{eq:Fv(-1)}),
\begin{align}
\tilde{z}_v([1])=-\alpha([1]^{\rm se}),  \qquad  \tilde{z}_v([-1])=-1.   \label{eq:z'(pm1)}
\end{align}

Consider the following table:
\begin{center}
\begin{tabular}{cccc}
  \hline
  $\xi^{\rm nw}$ & $\xi^{\rm ne}$ & $\xi^{\rm sw}$ & $\xi^{\rm se}$ \\
  \hline
  $1-b^{\rm sw}$ & $b^{\rm sw}$ & $-1$ & 0 \\
  $1-b^{\rm se}$ & $b^{\rm se}$ & 0 & $-1$ \\
  \hline
  $1-1/b^{\rm sw}$ & $-1$ & $1/b^{\rm sw}$ & 0  \\
  $1-b^{\rm se}/b^{\rm sw}$ & 0 & $b^{\rm se}/b^{\rm sw}$ & $-1$  \\
  \hline
\end{tabular}
\end{center}
The first and second rows correspond to $\xi^{\rm sw}=(1-b^{\rm sw})\xi^{\rm nw}+b^{\rm sw}\xi^{\rm ne}$ and
$\xi^{\rm se}=(1-b^{\rm se})\xi^{\rm nw}+b^{\rm se}\xi^{\rm ne}$ respectively; irrelevant columns have been omitted.
Then we can express $\xi^{\rm ne},\xi^{\rm se}$ as linear combinations of $\xi^{\rm nw},\xi^{\rm sw}$, corresponding to the third and fourth rows.
Note that the third row results from multiplying the first row by $-1/b^{\rm sw}$, and the fourth row is obtained by subtracting
$b^{\rm se}/b^{\rm sw}$ times the first row from the second row. Consequently,
\begin{align}
\tilde{z}_h(T)=-b^{\rm sw}\tilde{z}_v(T).   \label{eq:z-vs-b}
\end{align}

Moreover, from the process of proving (v), (h) we can see
\begin{align}
\tilde{z}_v(T_1\ast T_2)&=\tilde{z}_v(T_1)\tilde{z}_v(T_2),   \label{eq:zv-zv}  \\
\tilde{z}_h(T_1+T_2)&=\tilde{z}_h(T_1)\tilde{z}_h(T_2),    \label{eq:zh-zh}
\end{align}
since substitutions correspond to multiplying by matrices of determinant $1$.


Let $t_{\rm ne}$ denote the color of $T^{\rm ne}$, etc.
Observe that $T^{\rm nw}T^{\rm ne}T^{\rm se}T^{\rm sw}=1$ is a relation holding in $\pi_1(\mathfrak{B}\setminus T)$.
Hence $t_{\rm nw}t_{\rm ne}t_{\rm se}t_{\rm sw}=1$, and by (\ref{eq:trivial}),
\begin{align*}
0&=(1-t_{\rm nw})\xi^{\rm nw}+t_{\rm nw}(1-t_{\rm ne})\xi^{\rm ne}+t_{\rm nw}t_{\rm ne}(1-t_{\rm se})\xi^{\rm se}
+t_{\rm sw}^{-1}(1-t_{\rm sw})\xi^{\rm sw}  \\
&=(1-t_{\rm nw})\xi^{\rm nw}+t_{\rm nw}(1-t_{\rm ne})\xi^{\rm ne}
+t_{\rm sw}^{-1}t_{\rm se}^{-1}(1-t_{\rm se})(((1-b^{\rm se})\xi^{\rm nw}+b^{\rm se}\xi^{\rm ne})  \\
&\ \ \ \ +(t_{\rm sw}^{-1}-1)((1-b^{\rm sw})\xi^{\rm nw}+b^{\rm sw}\xi^{\rm ne})  \\
&=\big(1-t_{\rm nw}+t_{\rm sw}^{-1}(t_{\rm se}^{-1}-1)(1-b^{\rm se})+(t_{\rm sw}^{-1}-1)(1-b^{\rm sw})\big)\xi^{\rm nw} \\
&\ \ \ \ +\big((1-t_{\rm ne})t_{\rm nw}+t_{\rm sw}^{-1}(t_{\rm se}^{-1}-1)b^{\rm se}+(t_{\rm sw}^{-1}-1)b^{\rm sw}\big)\xi^{\rm ne}.
\end{align*}
Clearly, $\dim\{\mathbf{v}\in\mathbb{Q}(\Lambda)^{n+2}\colon Q_T\mathbf{v}=0\}\ge 2$.
Thus,
\begin{align*}
1-t_{\rm nw}+t_{\rm sw}^{-1}(t_{\rm se}^{-1}-1)(1-b^{\rm se})+(t_{\rm sw}^{-1}-1)(1-b^{\rm sw})=0,   \\
(1-t_{\rm ne})t_{\rm nw}+t_{\rm sw}^{-1}(t_{\rm se}^{-1}-1)b^{\rm se}+(t_{\rm sw}^{-1}-1)b^{\rm sw}=0.
\end{align*}
Both equations are equivalent to
\begin{align}
(1-t_{\rm se}^{-1})b^{\rm se}+(t_{\rm sw}-1)b^{\rm sw}=t_{\rm sw}t_{\rm nw}(1-t_{\rm ne})=t_{\rm se}^{-1}(t_{\rm ne}^{-1}-1).   \label{eq:constraint}
\end{align}
By (\ref{eq:c}), $b^{\rm sw}=1/c^{\rm ne}$, $b^{\rm se}=c^{\rm se}/c^{\rm ne}$, hence
\begin{align}
(t^{-1}_{\rm ne}-1)c^{\rm ne}+(1-t_{\rm se})c^{\rm se}=t_{\rm se}(t_{\rm sw}-1).   \label{eq:constraint'}
\end{align}

\begin{lem}\label{lem:composition}
Let $t_{\rm ne}$, $t_{\rm se}$, $t_{\rm sw}$ respectively denote the colors of $T_1^{\rm ne}$, $T_1^{\rm se}$, $T_1^{\rm sw}$.
Then in the notation of {\rm(\ref{eq:composition-b})}, {\rm(\ref{eq:composition-c})},
\begin{align}
\frac{\tilde{z}_h(T_1\ast T_2)}{\tilde{z}_v(T_1\ast T_2)}
=\frac{\tilde{z}_h(T_1)}{\tilde{z}_v(T_1)}\ast_{\mathsf{t}}\frac{\tilde{z}_h(T_2)}{\tilde{z}_v(T_2)},  \label{eq:fraction'-v}   \\
\frac{\tilde{z}_v(T_1+T_2)}{\tilde{z}_h(T_1+T_2)}
=\frac{\tilde{z}_v(T_1)}{\tilde{z}_h(T_1)}\circ_{\mathsf{t}}\frac{\tilde{z}_v(T_2)}{\tilde{z}_h(T_2)}.  \label{eq:fraction'-h}
\end{align}
\end{lem}

\begin{proof}
Remember that $-\tilde{z}_h(T)/\tilde{z}_v(T)=b_T^{\rm sw}$, the $(1,2)$-entry of $F_v^T$.

It follows from the self-evident identity $F_v^{T_1\ast T_2}=F_v^{T_2}F_v^{T_1}$ that
\begin{align*}
b^{\rm sw}_{T_1\ast T_2}&\ =b_{T_1}^{\rm sw}-b_{T_1}^{\rm sw}b_{T_2}^{\rm sw}+b_{T_1}^{\rm se}b_{T_2}^{\rm sw}    \\
&\stackrel{(\ref{eq:constraint})}=b_{T_1}^{\rm sw}-b_{T_1}^{\rm sw}b_{T_2}^{\rm sw}
+\frac{t_{\rm se}^{-1}(t_{\rm ne}^{-1}-1)+(1-t_{\rm sw})b_{T_1}^{\rm sw}}{1-t_{\rm se}^{-1}}b_{T_2}^{\rm sw}   \\
&\ =b_{T_1}^{\rm sw}+\frac{1-t_{\rm ne}^{-1}}{1-t_{\rm se}}b_{T_2}^{\rm sw}
+\frac{t_{\rm se}^{-1}-t_{\rm sw}}{1-t_{\rm se}^{-1}}b_{T_1}^{\rm sw}b_{T_2}^{\rm sw},
\end{align*}
which yields (\ref{eq:fraction'-v}).

Similarly, it follows from the identity $F_h^{T_1+T_2}=F_h^{T_2}F_h^{T_1}$ that
\begin{align*}
c^{\rm ne}_{T_1+T_2}&\ =c_{T_1}^{\rm ne}-c_{T_1}^{\rm ne}c_{T_2}^{\rm ne}+c_{T_1}^{\rm se}c_{T_2}^{\rm ne}   \\
&\stackrel{(\ref{eq:constraint'})}=c_{T_1}^{\rm ne}-c_{T_1}^{\rm ne}c_{T_2}^{\rm ne}
+\frac{t_{\rm se}(t_{\rm sw}-1)+(1-t_{\rm ne}^{-1})c_{T_1}^{\rm ne}}{1-t_{\rm se}}c_2^{\rm ne}   \\
&\ =c_{T_1}^{\rm ne}+\frac{1-t_{\rm sw}}{1-t_{\rm se}^{-1}}c_{T_2}^{\rm ne}
+\frac{t_{\rm se}-t_{\rm ne}^{-1}}{1-t_{\rm se}}c_{T_1}^{\rm ne}c_{T_2}^{\rm ne},
\end{align*}
which yields (\ref{eq:fraction'-h}).
\end{proof}

\begin{lem}\label{lem:integer}
For each $k\in\mathbb{Z}$,
\begin{enumerate}
  \item[\rm(i)] let $t_{\o}$ denote the color of $[k]^{\o}$ for $\o\in\{{\rm ne},{\rm se},{\rm sw}\}$, then
        \begin{align*}
        \tilde{z}_v([k])=\begin{cases} t_{\rm ne}(t_{\rm se}-1)[h]_{t_{\rm ne}t_{\rm se}},&k=2h  \\
        (1-t_{\rm se})[h]_{t_{\rm ne}t_{\rm se}}-1,&k=2h-1 \end{cases};
        \end{align*}
  \item[\rm(ii)] let $r_{\o}$ denote the color of $[1/k]^{\o}$ for $\o\in\{{\rm ne},{\rm se},{\rm sw}\}$, then
        \begin{align*}
        \tilde{z}_h([1/k])=\begin{cases} r_{\rm sw}^{-1}(r_{\rm se}^{-1}-1)[h]_{r_{\rm sw}^{-1}r_{\rm se}^{-1}},&k=2h  \\
        (1-r_{\rm se}^{-1})[h]_{r_{\rm sw}^{-1}r_{\rm se}^{-1}}-1,&k=2h-1 \end{cases}.
        \end{align*}
\end{enumerate}
\end{lem}

\begin{proof}
We only prove (i). The proof of (ii) is parallel.

Note that by definition, $\tilde{z}_v([k])$ is completely determined by $k$.

For $[2]=T_1+T_2$, with $T_1=T_2=[1]$, let $t_1=\alpha([2]^{\rm ne})=\alpha(T_1^{\rm se})$, and $t_2=\alpha([2]^{\rm se})=\alpha(T_1^{\rm ne})$.
By (\ref{eq:z'(pm1)}), $\tilde{z}_v(T_1)=-t_1$, $\tilde{z}_v(T_2)=-t_2$,
hence
$$\tilde{z}_v([2])=\tilde{z}_v(T_1)+\frac{1-t_2^{-1}}{1-t_1^{-1}}\tilde{z}_v(T_2)+\frac{t_2^{-1}-t_1}{1-t_1}\tilde{z}_v(T_1)\tilde{z}_v(T_2)
=t_1(t_2-1).$$
For $h\ge 2$ and $[2h]=T_1+T_2$, with $T_1=[2]$ and $T_2=[2h-2]$, let $t_1=\alpha([2h]^{\rm ne})=\alpha(T_1^{\rm ne})$, and
$t_2=\alpha([2h]^{\rm se})=\alpha(T_1^{\rm se})$. Then
\begin{align}
\tilde{z}_v([2h])&=\tilde{z}_v([2])+\tilde{z}_v([2h-2])+\frac{t_1^{-1}-t_2}{1-t_2}\tilde{z}_v([2])\tilde{z}_v([2h-2]) \nonumber  \\
&=\tilde{z}_v([2])+t_1t_2\tilde{z}_v([2h-2]),   \label{eq:integer-even-0}
\end{align}
so we can recursively deduce
\begin{align}
\tilde{z}_v([2h])=\frac{(t_1^ht_2^h-1)(t_2-1)}{t_2-t_1^{-1}}=t_1(t_2-1)[h]_{t_1t_2}.  \label{eq:integer-even}
\end{align}

For $\ell\ge 1$ and $[2\ell+1]=T_1+T_2$, with $T_1=[2\ell]$, $T_2=[1]$, let $t_1=\alpha(T_1^{\rm ne})=\alpha([2\ell+1]^{\rm se})$, and
$t_2=\alpha(T_1^{\rm se})=\alpha([2\ell+1]^{\rm ne})$. Then
\begin{align}
\tilde{z}_v([2\ell+1])&=\tilde{z}_v([2\ell])+\tilde{z}_v([1])+\frac{t_1^{-1}-t_2}{1-t_2}\tilde{z}_v([2\ell])\tilde{z}_v([1])  \nonumber  \\
&=t_1(t_2-1)[\ell]_{t_1t_2}-t_1^{\ell+1}t_2^\ell=(1-t_1)[\ell+1]_{t_1t_2}-1.     \label{eq:integer-odd}
\end{align}

For $[-2]=T_1+T_2$, with $T_1=T_2=[-1]$, by (\ref{eq:z'(pm1)}), $\tilde{z}_v(T_1)=\tilde{z}_v(T_2)=-1$, hence with
$t_1=\alpha(T_1^{\rm se})=\alpha([-2]^{\rm ne})$, and $t_2=\alpha(T_1^{\rm ne})=\alpha([-2]^{\rm se})$,
$$\tilde{z}_v([-2])=\tilde{z}_v(T_1)+\frac{1-t_2^{-1}}{1-t_1^{-1}}\tilde{z}_v(T_2)+\frac{t_2^{-1}-t_1}{1-t_1}\tilde{z}_v(T_1)\tilde{z}_v(T_2)
=t_2^{-1}-1.$$
For $h\ge 2$, let $t_1=\alpha([-2h]^{\rm ne})$, $t_2=\alpha([-2h]^{\rm se})$, then similarly as (\ref{eq:integer-even-0}),
$$\tilde{z}_v([-2h])=\tilde{z}_v([-2])+t_1^{-1}t_2^{-1}\tilde{z}_v([-(2h-2)]),$$
so we can recursively deduce
\begin{align*}
\tilde{z}_v([-2h])=\frac{(t_1^{-h}t_2^{-h}-1)(t_2-1)}{t_2-t_1^{-1}}=t_1(t_2-1)[-h]_{t_1t_2}.
\end{align*}
This can be incorporated with (\ref{eq:integer-even}).

For $\ell\ge 1$, let $t_1=\alpha([-2\ell-1]^{\rm se})$, and $t_2=\alpha([-2\ell-1]^{\rm ne})$, then similarly as $\tilde{z}_v([2\ell+1])$, we can obtain
\begin{align*}
\tilde{z}_v([-2\ell-1])&=\tilde{z}_v([-2\ell])+\tilde{z}_v([-1])+\frac{t_1^{-1}-t_2}{1-t_2}\tilde{z}_v([-2\ell])\tilde{z}_v([-1])  \\
&=\frac{(t_1^{-\ell}t_2^{-\ell}-1)(t_2-1)}{t_2-t_1^{-1}}+t_1^{-\ell}t_2^{-\ell}
=(1-t_1)[-\ell]_{t_1t_2}-1,
\end{align*}
which together with (\ref{eq:integer-odd}) verify the formula for $\tilde{z}_v([2h-1])$, $h\in\mathbb{Z}$.
\end{proof}

As a consequence of Lemma \ref{lem:composition} and Lemma \ref{lem:integer}, for generic $T$,
$$[\tilde{z}_v(T):\tilde{z}_h(T)]=[z_v(T):z_h(T)]=z(T),$$
where $z(T)$ is introduced in Section 2.

\subsection{Turning back to $Q_L$}

\begin{proof}[Proof of Theorem \ref{thm:main} in the generic case]
As we have shown, $Q_T$ can be transformed into $N_v^T=UQ_TP$, where $U=U_v^T$, and $P=P_v^T$ is a permutation matrix used to arrange the columns so that the last two columns of $N_v^T$ correspond to $T^{\rm sw}$, $T^{\rm se}$.

Observe that $UQ_L$ can be obtained from $UQ_T$ by deleting the two columns corresponding to $T^{\rm sw}$, $T^{\rm se}$ and
modifying the last two rows
according to $\xi^{\rm nw}=\xi^{\rm sw}$ and $\xi^{\rm ne}=\xi^{\rm se}$, because in $L$, the underlying arcs of $T^{\rm nw}$ and $T^{\rm sw}$ are identified, so are those of $T^{\rm ne}$ and $T^{\rm se}$.
Hence for certain permutation matrix $P'$, in block form we have
$$UQ_LP'=\left(\begin{array}{ccc} \star & \star & -I_{n-2} \\
-b^{\rm sw} & b^{\rm sw} & 0 \\ 1-b^{\rm se} & b^{\rm se}-1 & 0 \end{array}\right),$$
where each $\star$ stands for a column that is irrelevant, and $I_{n-2}$ denotes the $(n-2)\times(n-2)$ identity matrix. Deleting the last row (which stems from the crossing $\mathfrak{c}_{i_0}$) and the first column, we obtain that for certain $j_0$,
\begin{align}
\det\big((Q_L)^{\neg j_0}_{\neg i_0}\big)\doteq \det(U^{-1})\cdot b^{\rm sw}
=\tilde{z}_v(T)b^{\rm sw}\stackrel{(\ref{eq:z-vs-b})}=-\tilde{z}_h(T)\doteq-z_h(T).   \label{eq:det}
\end{align}
Combining this with (\ref{eq:AP}) finishes the proof.
\end{proof}

To deal with the nongeneric case, we use a ``deformation technique".

For $S\in\mathcal{T}_{\rm ar}$, define $f(S)\in\mathbb{Q}\cup\{0,\infty\}$ recursively by $f([p/q])=p/q$ and
$$f(S_1+S_2)=f(S_1)+f(S_2), \qquad    \frac{1}{f(S_1\ast S_2)}=\frac{1}{f(S_1)}+\frac{1}{f(S_2)},$$
with the convention that $\infty+\infty=\infty$, $1/0=\infty$, $1/\infty=0$.

Suppose $f(S)\ne 0,\infty$ for all subtangles $S$ subsequent to $T$.
When all $t_i=-1$, by Remark \ref{rmk:rational}, $z_v(p/q)=p$, $z_h(p/q)=q$; using (\ref{eq:fraction'-v}), (\ref{eq:fraction'-h}) we can recursively show $z_v(S)/z_h(S)=f(S)$. Consequently, $z_v(S)z_h(S)\ne 0$ for each subsequent subtangle $S$ of $T$, i.e. the generic condition is fulfilled.

\begin{figure}[h]
  \centering
  \includegraphics[width=11cm]{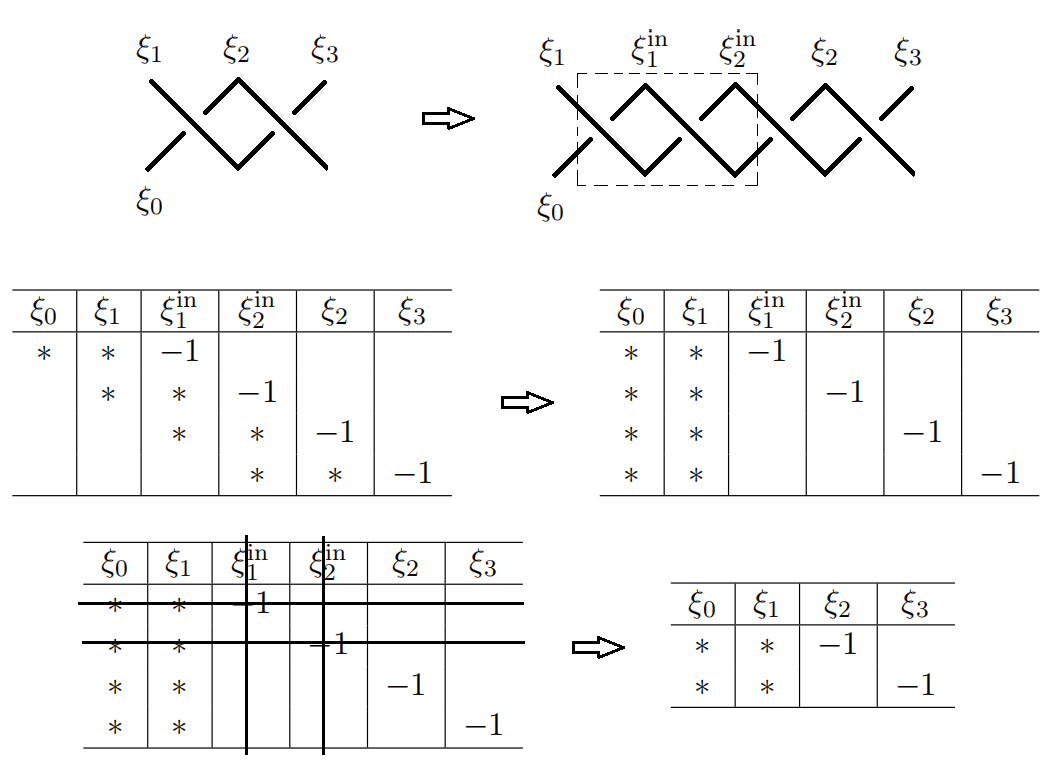}\\
  \caption{Upper: insert $[2a_i]$ (as bounded by the dashed lines) into the integer part of $R_i$; we use $\xi_i$, etc. to label the corresponding arc. Middle: use row transformations to eliminate the elements in each column corresponding to an arc in the inserted tangle $[2a_i]$. Lower: delete the rows and columns containing the $-1$'s, then the determinant does not change.}\label{fig:insert}
\end{figure}

Let $R_1,\ldots,R_d$ be the rational tangles subsequent to $T$.
Given any $\vec{a}=(a_1,\ldots,a_d)\in\mathbb{Z}^d$, insert $[2a_i]$ into the integer part of $R_i$ for each $i$.
Let $T_{\vec{a}}$ denote the resulting tangle, and let $L_{\vec{a}}=D(T_{\vec{a}})$. Then $L_{\vec{a}}=L$ if $\vec{a}=0$.
Modify $Q_{L_{\vec{a}}}$ as follows: for each $i$, use row transformations to reduce the nonzero elements in the column corresponding to each arc in $[2a_i]$, leaving only $-1$, then delete the row and the column containing it.
The case $a_i=1$ is illustrated in Figure \ref{fig:insert}.
Let $V_{\vec{a}}$ denote the resulting matrix, which has the same size as $Q_L$.

Clearly, $T_{\vec{a}}$ is generic for generic $\vec{a}$.
What has been shown is that there exists a unit $\kappa$ independent of $\vec{a}$ such that
$$G(\vec{a}):=\det\big((V_{\vec{a}})^{\neg j_0}_{\neg i_0}\big)+\kappa z_h(T_{\vec{a}})=0$$
for generic $\vec{a}$. From Lemma \ref{lem:integer} and its proof we can see that $G(\vec{a})$ can be regarded as a polynomial in $[a_1]_{s_1},\ldots,[a_d]_{s_d}$, where $s_i$ has the form $t_{j}^{\mu}t_{k}^{\nu}$ for $\mu,\nu\in\{\pm1\}$.
Thus, actually $G(\vec{a})=0$ for all $\vec{a}\in\mathbb{Z}^d$.
The case $\vec{a}=0$ reads
$$\det\big((Q_L)^{\neg j_0}_{\neg i_0}\big)\doteq-z_h(T).$$
This finishes the proof of Theorem \ref{thm:main}.

\section*{Declarations and statements}

{\bf Funding}: No funding was received for conducting this study.

\noindent
{\bf Conflict of interest}: The author has no conflict of interest to declare, 
and has no relevant financial or non-financial interests to disclose.

\noindent
{\bf Data availability statement}: Data sharing not applicable to this article as no data sets were generated or analyzed during the current study.

\bigskip

\noindent
Haimiao Chen (orcid: 0000-0001-8194-1264)\ \ \ \ {\it chenhm@math.pku.edu.cn} \\
Department of Mathematics, Beijing Technology and Business University, \\
Liangxiang Higher Education Park, Fangshan District, Beijing, China.


\begin{thebibliography}{}


\bibitem{Al28}
J.W. Alexander,
Topological invariants of knots and links,
{\it Trans. Amer. Math. Soc.} {\bf 30} (1928), no. 2, 275--306.


\bibitem{BL20}
Y. Bae, I.S. Lee,
On Alexander polynomials of pretzel links,
{\it Kyungpook Math. J.} {\bf 60} (2020), 239--253.


\bibitem{Be25}
Y. Belousov,
Explicit formulas for the Alexander polynomial of pretzel knots,
arXiv: 2502.10370.


\bibitem{BS16}
F. Bonahon and L. Siebenmann,
{\it New geometric splittings of classical knots and the classification and symmetries of arborescent knots},
unpublished manuscript (2016).


\bibitem{Ch21}
H.-M. Chen, Computing twisted Alexander polynomials for Montesinos links,
{\it Indian J. Pure Appl. Math.} {\bf 52} (2021), 584--598.


\bibitem{FP16}
K. Finlinson, J.S. Purcell,
Volumes of Montesinos links,
{\it Pac. J. Math.} 282 (2016), no. 1, 63--105.


\bibitem{Ha79}
R.I. Hartley,
On two-bridged knot polynomials,
{\it J. Austral. Math. Soc. (Ser. A)} {\bf 28} (1979), no. 2, 241--249.


\bibitem{Hi01}
E. Hironaka,
The Lehmer polynomial and pretzel links,
{\it Canad. Math. Bull.} {\bf 44} (2001), no. 4, 440--451.


\bibitem{HM19}
M. Hirasawa, K. Murasugi,
Stable Alexander polynomials of arborescent links,
{\it J. Knot Theory Ramifications} {\bf 28} (2019), no. 13, 1940017, 21 pp.


\bibitem{Ho20}
J. Hoste,
A note on Alexander polynomials of 2-bridge links,
{\it J. Knot Theory Ramifications} {\bf 29} (2020), no. 8, 1971003, 7 pp.



\bibitem{Ka84}
T. Kanenobu,
Alexander polynomials of two-bridge links,
{\it J. Austral. Math. Soc. (Series A)} {\bf 36} (1984), 59--68.


\bibitem{KL07}
D. Kim, J. Lee,
Some invariants of pretzel links,
{\it Bull. Austral. Math. Soc.} {\bf 75} (2007), 253--271.



\bibitem{Li97}
W.B.R. Lickorish,
{\it An introduction to knot theory},
Graduate Texts in Mathematics, vol. {\bf 175}, Springer New York, 1997.



\bibitem{Na86}
Y. Nakagawa,
On the Alexander polynomials of the pretzel links $L(p_1,\ldots,p_n)$,
{\it Kobe J. Math.} {\bf 3} (1986), 167--177.

\end{thebibliography}
\end{document}